\documentclass[reqno]{amsart}

\usepackage{amssymb,amsmath,amsthm,xcolor,enumerate,hyperref,cleveref}

\usepackage{tikz-cd}
\usetikzlibrary{positioning}

\allowdisplaybreaks

\newtheorem{thrm}{Theorem}[section]
\newtheorem{cor}[thrm]{Corollary}
\newtheorem{lem}[thrm]{Lemma}
\newtheorem{prop}[thrm]{Proposition}

\theoremstyle{definition}
\newtheorem{defn}[thrm]{Definition}

\newtheorem{rem}[thrm]{Remark}

\crefrangeformat{equation}{#3(#1)#4--#5(#2)#6}

\crefname{thrm}{Theorem}{Theorems}
\crefname{lem}{Lemma}{Lemmas}
\crefname{cor}{Corollary}{Corollaries}
\crefname{prop}{Proposition}{Propositions}
\crefname{defn}{Definition}{Definitions}
\crefname{exm}{Example}{Examples}
\crefname{rem}{Remark}{Remarks}
\crefname{section}{Section}{Sections}
\crefname{equation}{\unskip}{\unskip}
\crefname{enumi}{\unskip}{\unskip}

\makeatletter
\newcommand{\mylabel}[2]{#2\def\@currentlabel{#2}\label{#1}}
\makeatother
\renewcommand{\iff}{\Leftrightarrow}

\newcommand{\Mod}[1]{{\sf Mod}{(#1)}}

\newcommand{\impl}{\Rightarrow}

\newcommand{\id}{\mathrm{id}}

\newcommand{\cIui}[1]{{\mathcal I}_{ui}{(#1)}}

\newcommand{\cD}{\mathcal D}

\newcommand{\cU}[1]{\mathcal U{(#1)}}

\newcommand{\cG}[1]{\mathcal G{(#1)}}

\newcommand{\cM}[1]{\mathcal M{(#1)}}

\newcommand{\bb}[1]{\mathbb #1}

\newcommand{\End}[1]{\operatorname{\mathrm{End}}{#1}}

\newcommand{\iend}[1]{\operatorname{\mathit{end}}{#1}}

\DeclareMathOperator{\im}{im}%

\newcommand{\Hom}{\mathrm{Hom}}

\newcommand{\m}{{}^{-1}}

\newcommand{\0}{\theta}
\newcommand{\e}{\varepsilon}

\begin{document}
	
	\title[Partial cohomology and extensions]{Partial cohomology of groups and extensions of semilattices of abelian groups}
	
	\author{Mikhailo Dokuchaev}
	\address{Insituto de Matem\'atica e Estat\'istica, Universidade de S\~ao Paulo,  Rua do Mat\~ao, 1010, S\~ao Paulo, SP,  CEP: 05508--090, Brazil}
	\email{dokucha@gmail.com}
	
	\author{Mykola Khrypchenko}
	\address{Departamento de Matem\'atica, Universidade Federal de Santa Catarina, Campus Reitor Jo\~ao David Ferreira Lima, Florian\'opolis, SC,  CEP: 88040--900, Brazil}
	\email{nskhripchenko@gmail.com}
	
	
	\thanks{The first author was partially supported by CNPq of Brazil (Proc. 305975/2013--7) and by FAPESP of Brazil (Proc. 2015/09162--9). The second author was partially supported by FAPESP of Brazil (Proc. 2012/01554--7).}

	\begin{abstract}
		We extend the notion of a partial cohomology group $H^n(G,A)$ to the case of non-unital $A$ and find interpretations of $H^1(G,A)$ and $H^2(G,A)$ in the theory of extensions of semilattices of abelian groups by groups.
	\end{abstract}
	
	\maketitle
	
	\section*{Introduction}\label{sec-intro}
	
	Elaborated  in the theory of $C^*$-algebras   the  concept   of a  partial action draws growing attention of experts in analysis, algebra and beyond, resulting in new theoretic advances and remarkable applications. The early developments on partial actions and related concepts are described in the short survey \cite{D}, whereas the algebraic and  $C^*$-algebraic foundations of the theory, a detailed treatment of graded $C^*$-algebras by means of Fell bundles and partial $C^*$-crossed products, as well as prominent applications are the contents of the recent Exel's book \cite{E6} (see also the surveys \cite{F2,Pi4}).
	
	Exel's notion of a twisted partial group action \cite{E0} (see also \cite{DES1}) involves a kind of $2$-cocycle equality which suggested the development of a  cohomology theory  based on partial actions. As a first step, the theory of partial projective representations of groups was created in \cite{DN,DN2,DoNoPi} with further results  in \cite{DdLP,NP,Pi,Pi4}.  As expected, the notion of the corresponding Schur Multiplier, which is a semilattice of abelian groups, appeared in the treatment. While the idea of partial $2$-cocycles came from  \cite{E0,DES1}, the partial coboundaries naturally appeared in the notion of an equivalence of twisted partial actions introduced in \cite{DES2} with respect to the globalization problem.  The general definition of partial cohomology groups was given in \cite{DK}, where they were related to H. Lausch's cohomology of inverse semigroups \cite{Lausch}, and it was shown, moreover, that  each component of the partial Schur Multiplier is a disjoint union of cohomology groups with values in non-necessarily trivial partial modules. In \cite{DK} we assumed that the partial actions under consideration are all unital, which is a natural working restriction  made in the big majority of algebraic papers dealing with partial actions.

	Our subsequent paper \cite{DK2} was stimulated by the desire  to relate the second partial cohomology groups with extensions. During the course of the investigation it became clear that  the restriction on a partial action to be unital  may be omitted. The basic concept is that of an extension of a semilattice of groups $A$ by a group $G$, the main example being  the crossed product  $A\ast_\Theta G$ by a twisted partial action $\Theta$ of $G$ on $A$, and we explored a relation of these notions to  extensions of $A$ by an inverse semigroup $S,$  twisted $S$-modules and the corresponding crossed products in the sense of \cite{Lausch}. In fact, we deal with twisted $S$-modules structures on $A,$ whose twistings satisfy a normality condition, considered by N. Sieben in   \cite{Sieben98}, which is stronger than the one imposed by H. Lausch in \cite{Lausch}, and   we call them Sieben's twisted modules.   One of the main results of \cite{DK2}  establishes, up to certain equivalences and identifications, a  one-to-one correspondence between twisted partial actions of groups on $A$ and Sieben's twisted module structures  on $A$ over $E$-unitary inverse semigroups.
	
	In the present paper we define the cohomology groups with values in a non-necessarily unital partial module and give interpretations for the first and second cohomology groups in terms of extensions of semilattices of abelian groups by groups. In \cref{sec-prelim} we recall  some background from \cite{Clifford-Preston-1,DE,DES1,DES2,DK2,Howie,Lausch,Lawson} on inverse semigroups,  their cohomology, partial group actions and related notions  needed in the sequel. We start  \cref{sec-H^n(S_A)} by revising the free resolutions  $C(S)$ and $D(S)$ of $\bb Z_S$ from \cite{Lausch} and correct a couple of errors with respect to   $C(S)$ made in  \cite{Lausch} and give some details which are missing in  \cite{Lausch}   (see, in particular, \cref{rem-Lausch's-partial'',rem-<sigma'n(C_n(S))>-ne-C-(n+1)(S),rem-sigma'_n(C_n(S))-ne-C_(n+1)(S)}).  Next, since  Sieben's twisted $S$-modules have order preserving twistings, we define in the same \cref{sec-H^n(S_A)} the cohomology groups $H^n_\le(S^1,A^1)$ based on order preserving cochains and prove some basic facts (see \cref{sec-H^n_<=}). The cohomology groups   $H^n(G,A)$ with values in a non-necessarily unital partial $G$-module $A$ are defined in \cref{sec-non-unital-H^n}, some preliminary facts are proved and a relation to  $H^n_\le(S^1,A^1)$ is established (see \cref{H^n(GA)-cong-H^n_<=(SA)}). The latter, in its turn, implies a connection of $H^2(G,A)$ with the equivalence classes  of twistings related to $A$ (see \cref{w-cohom-to-twisting}). An interpretation of $H^2(G,A)$ in terms of extensions of semilattices of abelian groups by $G$ is given in \cref{sec-ext-abel}, the main result being \cref{cl-eq-ext<->H^2}. This is done in close interaction with extensions of semilattices of groups by inverse semigroups studied in~\cite{Lausch}. In the final \cref{sec-split} we relate split extensions (see \cref{A->U->G-splits-defn}) and $H^1(G,A)$. This is done in two steps. First, given an inverse semigroup $S$ and an $S$-module $A$, we consider split extensions $A \to U \to S$  and prove in \cref{conjugate-classes<->H^1_<=(S^1_A^1)} that  the so-called  $C^0_\le$-equivalence classes of splittings of $U$ are in a one-to-one correspondence with the elements of $H^1_\le(S^1,A^1)$. Then, for a group $G$ and a $G$-module $A,$ we define the concept of a split extension $A \to U\to G$ (see \cref{A->U->G-splits-defn}),  and  using  \cref{conjugate-classes<->H^1_<=(S^1_A^1),H^n(GA)-cong-H^n_<=(SA)} we show in \cref{eq-classes-of-splittings<->H^1(G_A)} that the equivalence classes of splittings of $U$ are in a one-to-one correspondence with the elements of $H^1(G,A)$.

	\section{Preliminaries}\label{sec-prelim}
	
	\subsection{Inverse semigroups}\label{subsec-inv-sem-congr}
	
	A semigroup $S$ is called {\it regular}, whenever for each $s\in S$ there exists $t\in S$, called an {\it inverse of} $s$, such that $sts=s$ and $tst=t$. Regular semigroups, in which every element $s$ admits a unique inverse, usually denoted by $s\m$, are called {\it inverse semigroups}. These are precisely those regular semigroups, whose idempotents commute (see~\cite[Theorem 1.1.3]{Lawson}). 
	
	Each inverse semigroup $S$ admits the natural partial order $\le$ with $s\le t$ whenever $s=et$ for some $e\in E(S)$ (see~\cite[p. 21]{Lawson}), where $E(S)$ denotes the semilattice of idempotents of $S$. It follows that the binary relation $(s,t)\in\sigma\iff\exists u\le s,t$ is a group congruence on $S$ in the sense that $S/\sigma$ is a group (see~\cite[2.4]{Lawson}). Moreover, it is the {\it minimum group congruence} on $S$, since it is contained in any other such congruence~\cite[p. 62]{Lawson}. The quotient $S/\sigma$ is thus called the {\it maximum group image} of $S$ and denoted by $\cG S$. 
	
	It is easy to see that all idempotents of $S$ are $\sigma$-equivalent, since $ef\le e,f$ for $e,f\in E(S)$. Inverse semigroups, in which idempotents constitute a $\sigma$-class, are called {\it $E$-unitary}~\cite[p. 64]{Lawson}. Equivalently, $S$ is $E$-unitary if, given $s\in S$ and $e\in E(S)$, it follows from $e\le s$ that $s\in E(S)$. Another property that characterizes $E$-unitary inverse semigroups: $(s,t)\in\sigma\iff s\m t,st\m\in E(S)$~\cite[Theorem 2.4.6]{Lawson}.
	
	A {\it semilattice of groups} is an inverse semigroup $A$ which can be represented as a disjoint union of groups\footnote{Such inverse semigroups are also called Clifford semigroups (see~\cite[IV.2]{Howie} and~\cite[5.2]{Lawson}).}, more precisely, $A=\bigsqcup_{e\in E(A)}A_e$, where $A_e=\{a\in A\mid aa\m=a\m a=e\}$. For an inverse semigroup $A$ each of the following conditions is equivalent to the fact that $A$ is a semilattice of groups:
	\begin{enumerate}
		\item $aa\m=a\m a$ for all $a\in A$,
		\item $E(A)\subseteq C(A)$.
	\end{enumerate}
	In particular, any commutative inverse semigroup is a semilattice of (abelian) groups.
	
	\subsection{Twisted partial actions of groups on semigroups}\label{subsec-tw-pact}
	
	Recall from~\cite{DE,DK2} that a {\it multiplier} of a semigroup $S$ is a pair $w$ of maps $s\mapsto ws$ and $s\mapsto sw$ from $S$ to itself, such that for all $s,t\in S$
	\begin{enumerate}
		\item $w(st)=(ws)t$;
		\item $(st)w=s(tw)$;
		\item $s(wt)=(sw)t$.\footnote{Observe that $w$ is exactly a pair of linked right and left translations~\cite[p. 10]{Clifford-Preston-1}.}
	\end{enumerate}
	The multipliers of $S$ form a monoid $\cM S$ under the composition (see~\cite{DE} and~\cite[p. 11]{Clifford-Preston-1}). If $S$ is inverse, then 
	\begin{align}\label{eq-ws-inv}
	(ws)\m=s\m w\m,\ \ (sw)\m=w\m s\m
	\end{align}
	for all $s\in S$ and $w\in\cU{\cM S}$. Here and below $\cU{M}$ denotes the group of invertible elements of a monoid $M$.
	
	A {\it twisted partial action}~\cite{DES1,DK2} of a group $G$ on a semigroup $S$ is a pair $\Theta=(\0,w)$, where $\0$ is a collection $\{\0_x:\cD_{x\m}\to \cD_x\}_{x\in G}$ of isomorphisms between non-empty ideals of $S$ and $w=\{w_{x,y}\in\cU{\cM{\cD_x\cD_{xy}}}\}_{x,y\in G}$, such that for all $x,y,z\in G$
	\begin{enumerate}
		\item $\cD_x^2=\cD_x$ and $\cD_x\cD_y=\cD_y\cD_x$;\label{D_x^2=D_x}
		\item $\cD_1=S$ and $\0_1=\id_S$;\label{0_1=id_S}
		\item $\0_x(\cD_{x\m}\cD_y)=\cD_x\cD_{xy}$;\label{0_x(D_(x^(-1))D_y)}
		\item $\0_x\circ\0_y(s)=w_{x,y}\0_{xy}(s)w\m_{x,y}$ for any $s\in\cD_{y\m}\cD_{y\m x\m}$;\label{0_x-circ-0_y(s)}
		\item $w_{1,x}=w_{x,1}=\id_{\cD_x}$;
		\item $\0_x(sw_{y,z})w_{x,yz}=\0_x(s)w_{x,y}w_{xy,z}$ for all $s\in\cD_{x\m}\cD_y\cD_{yz}$.\label{0_x(sw_yz)w_xyz}
	\end{enumerate}
	Here the right-hand side of \cref{0_x-circ-0_y(s)} does not require a pair of additional brackets, since $\cD_x\cD_{xy}$ is an idempotent ideal and thus due to (ii) of~\cite[Proposition 2.5]{DE} one has $(ws)w'=w(sw')$ for all $w,w'\in\cM{\cD_x\cD_{xy}}$ and $s\in \cD_x\cD_{xy}$. The applicability of the multipliers in \cref{0_x(sw_yz)w_xyz} is explained by \cref{D_x^2=D_x,0_x(D_(x^(-1))D_y)}.
	
	Observe that, when $S$ is inverse, each ideal $I$ of $S$ is idempotent, as for $s\in I$ one has $s=s\cdot s\m s$ with $s\m s\in I$. It follows that $I\cap J=IJ$ of any two non-empty ideals of $S$, since $I\cap J=(I\cap J)^2\subseteq IJ$. In particular, any two non-empty ideals of $S$ commute. This shows that for an inverse $S$ item \cref{D_x^2=D_x} can be dropped and in \cref{0_x(D_(x^(-1))D_y),0_x-circ-0_y(s),0_x(sw_yz)w_xyz} the products of domains can be replaced by their intersections.
	
	By a {\it partial action}~\cite{DE} of $G$ on $S$ we mean a collection $\0=\{\0_x:\cD_{x\m}\to \cD_x\}_{x\in G}$ as above satisfying
	\begin{enumerate}
		\item $\cD_1=S$ and $\0_1=\id_S$;
		\item $\0_x(\cD_{x\m}\cap\cD_y)=\cD_x\cap\cD_{xy}$;
		\item $\0_x\circ\0_y(s)=\0_{xy}(s)$ for any $s\in\cD_{y\m}\cap\cD_{y\m x\m}$.
	\end{enumerate}
	When $S$ is inverse, this is exactly a twisted partial action of $G$ on $S$ with trivial $w$. 
	
	Two twisted partial actions $(\0,w)$ and $(\0',w')$ of $G$ on $S$ are called {\it equivalent}~\cite{DES2,DK2}, if for all $x\in G$
	\begin{enumerate}
		\item $\cD'_x=\cD_x$
    \end{enumerate}
	and there exists $\{\e_x\in\cU{\cM{\cD_x}}\mid x\in G\}$ such that for all $x,y\in G$
	\begin{enumerate}
		\item[(ii)] $\0'_x(s)=\e_x\0_x(s)\e\m_x$, $s\in\cD_{x\m}$;
		\item[(iii)] $\0'_x(s)w'_{x,y}\e_{xy}=\e_x\0_x(s\e_y)w_{x,y}$, $s\in\cD_{x\m}\cD_y$.
	\end{enumerate}
	
	Given a twisted partial action $\Theta=(\theta,w)$ of $G$ on $S$, the crossed product $S*_\Theta G$ is the set $\{s\delta_x\mid s\in \cD_x\}$ with the multiplication $s\delta_x\cdot t\delta_y=\0_x(\0\m_x(s)t)w_{x,y}\delta_{xy}$. It follows from the proof of~\cite[Theorem 2.4]{DES1} that $S*_\Theta G$ is a semigroup. Moreover, if $S$ is inverse, then $S*_\Theta G$ is inverse with $(s\delta_x)\m=w\m_{x\m,x}\0_{x\m}(s\m)\delta_{x\m}$ (see~\cite{DK2}). For the crossed product coming from a partial action $\0$ of $G$ on $S$ we shall use the notation $S*_\0 G$.
	
	\subsection{Twisted modules over inverse semigroups}\label{subsec-tw-mod}
	
	An endomorphism $\varphi$ of a semilattice of groups $A$ is said to be {\it relatively invertible}~\cite{Lausch}, whenever there exist $\bar\varphi\in\End A$ and $e_\varphi\in E(A)$ satisfying
	\begin{enumerate}
		\item $\bar\varphi\circ\varphi(a)=e_\varphi a$ and $\varphi\circ\bar\varphi(a)=\varphi(e_\varphi)a$ for any $a\in A$; 
		\item $e_\varphi$ is the identity of $\bar\varphi(A)$ and $\varphi(e_\varphi)$ is the identity of $\varphi(A)$.
	\end{enumerate}
	In this case $\bar\varphi$ is also relatively invertible with $\bar{\bar\varphi}=\varphi$ and $e_{\bar\varphi}=\varphi(e_\varphi)$.  The set of relatively invertible endomorphisms forms an inverse semigroup $\iend A$~\cite[Proposition 8.1]{Lausch}. It was proved in~\cite{DK2} that $\iend A$ is isomorphic to the semigroup $\cIui A$ of isomorphisms between unital ideals of $A$.
	
	Let $S$ be an inverse semigroup. By a {\it twisted $S$-module}~\cite{Lausch,DK2} we mean a semilattice of groups $A$ together with a triple $\Lambda=(\alpha,\lambda,f)$, where $\alpha$ is an isomorphism $E(S)\to E(A)$, $\lambda$ is a map $S\to\iend A$ and $f:S^2\to A$ is a map with $f(s,t)\in A_{\alpha(stt\m s\m)}$ (called a {\it twisting}) satisfying the following properties 
	\begin{enumerate}
		\item $\lambda_e(a)=\alpha(e)a$ for all $e\in E(S)$ and $a\in A$;\label{lambda_e(a)}
		\item $\lambda_s(\alpha(e))=\alpha(ses\m)$ for all $s\in S$ and $e\in E(S)$;\label{lambda_s(alpha(e))}
		\item $\lambda_s\circ\lambda_t(a)=f(s,t)\lambda_{st}(a)f(s,t)\m$ for all $s,t\in S$ and $a\in A$;\label{lambda_s-circ-lambda_t(a)}
		\item $f(se,e)=\alpha(ses\m)$ and $f(e,es)=\alpha(ess\m)$ for all $s\in S$ and $e\in E(S)$;\label{f(se_e)-and-f(e_es)}
		\item $\lambda_s(f(t,u))f(s,tu)=f(s,t)f(st,u)$ for all $s,t,u\in S$.\label{2-cocycle-id-for-f}
	\end{enumerate}
	
	If $A$ is commutative, then a twisted $S$-module $\Lambda=(\alpha,\lambda,f)$ on $A$ splits into the {\it $S$-module} $(\alpha,\lambda)$ on $A$ in the sense of~\cite[p. 274]{Lausch}, that is the one satisfying \cref{lambda_e(a),lambda_s(alpha(e))} with $\lambda$ being a homomorphism $S\to\End A$, and the map $f(s,t)\in A_{\alpha(stt\m s\m)}$, for which \cref{f(se_e)-and-f(e_es),2-cocycle-id-for-f} hold. Such a map $f$ will be called a {\it twisting related to $(\alpha,\lambda)$}. 
	
	Observe from~\cite{Lausch} that $S$-modules form an abelian category $\Mod S$, where a morphism $\varphi:(\alpha,\lambda)\to(\alpha',\lambda')$ is a homomorphism of semigroups, such that 
	\begin{enumerate}
		\item $\varphi\circ\alpha=\alpha'$ on $E(S)$;\label{varphi-alpha=alpha'}
		\item $\varphi\circ\lambda_s=\lambda_s'\circ\varphi$ for all $s\in S$.\label{varphi-lambda_s=lambda'_s-varphi}
	\end{enumerate}
	
	Two twisted $S$-module structures $\Lambda=(\alpha,\lambda,f)$ and $\Lambda'=(\alpha',\lambda',f')$ on $A$ are said to be {\it equivalent}~\cite{DK2}, if
	\begin{enumerate}
		\item $\alpha'=\alpha$;\label{eq-tw-mod-alpha'}
		\item $\lambda'_s(a)=g(s)\lambda_s(a)g(s)\m$;\label{eq-tw-mod-lambda'}
		\item $f'(s,t)g(st)=g(s)\lambda_s(g(t))f(s,t)$\label{eq-tw-mod-f'}
	\end{enumerate}
	for some function $g:S\to A$ with $g(s)\in A_{\alpha(ss\m)}$. When $A$ is commutative, this exactly means that $(\alpha,\lambda)=(\alpha',\lambda')$ and the twistings $f$ and $f'$ are {\it equivalent} in the sense that $f'(s,t)=\lambda_s(g(t))g(st)\m g(s)f(s,t)$.
	
	Let $\Lambda=(\alpha,\lambda,f)$ be a twisted $S$-module structure on $A$. The {\it crossed product of $A$ and $S$ by $\Lambda$}~\cite{Lausch,DK2} is the set 
	$$
	A*_\Lambda S=\{a\delta_s\mid a\in A,\ \ s\in S,\ \ aa\m=\alpha(ss\m)\}.
	$$
	It is an inverse semigroup under the multiplication $a\delta_s\cdot b\delta_t=a\lambda_s(b)f(s,t)\delta_{st}$ with 
	\begin{align}\label{a-delta_s-inv}
	(a\delta_s)\m=f(s\m,s)\m\lambda_{s\m}(a\m)\delta_{s\m} 
	\end{align}
	(see~\cite[Theorem 9.1]{Lausch}). If $(\alpha,\lambda)$ is an $S$-module structure on $A$, then $A*_{(\alpha,\lambda)}S$ will mean the crossed product of $A$ and $S$ by $(\alpha,\lambda,f)$ with trivial $f$. 
	
	\subsection{Extensions of semilattices of groups by inverse semigroups}\label{subsec-ext-of-A-by-S}
	An {\it extension} of a semilattice of groups $A$ by an inverse semigroup $S$~\cite{Lausch} is an inverse semigroup $U$ with a monomorphism $i:A\to U$ and an idempotent-separating (i.\,e. injective on $E(U)$) epimorphism $j:U\to S$, such that $i(A)=j\m(E(S))$. 
	
	Any two extensions $A\overset{i}{\to}U\overset{j}{\to}S$ and $A\overset{i'}{\to}U'\overset{j'}{\to}S$ of $A$ by $S$ are called {\it equivalent}~\cite{Lausch} if there is a homomorphism $\mu:U\to U'$ such that the following diagram
	\begin{align*}
	\begin{tikzpicture}[node distance=1.5cm, auto]
	\node (A) {$A$};
	\node (U) [right of=A] {$U$};
	\node (S) [right of=U] {$S$};
	\node (A') [below of=A]{$A$};
	\node (U') [below of=U] {$U'$};
	\node (S') [below of=S] {$S$};
	\draw[->] (A) to node {$i$} (U);
	\draw[->] (U) to node {$j$} (S);
	\draw[->] (A') to node {$i'$} (U');
	\draw[->] (U') to node {$j'$} (S');
	\draw[-,double distance=2pt] (A) to node {} (A');
	\draw[->] (U) to node {$\mu$} (U');
	\draw[-,double distance=2pt] (S) to node {} (S');
	\end{tikzpicture}
	\end{align*}
	commutes. In this case $\mu$ is an isomorphism.
	
	For each twisted $S$-module structure $\Lambda=(\alpha,\lambda,f)$ on $A$ the crossed product $A*_\Lambda S$ is an extension of $A$ by $S$, where $i(a)=a\delta_{\alpha\m(aa\m)}$ and $j(a\delta_s)=s$.
	
	Given an extension $A\overset{i}{\to}U\overset{j}{\to}S$, a map $\rho:S\to U$ with $j\circ\rho=\id$ and $\rho(E(S))\subseteq E(U)$ is called a {\it transversal}~\cite{Lausch} of $j$. It follows that $\rho$ maps isomorphically $E(S)$ onto $E(U)$ and 
	\begin{align}\label{rho|_E(S)=j|_E(U)-inv}
	\rho|_{E(S)}=(j|_{E(U)})\m.
	\end{align}
	
	The choice of a transversal $\rho$ induces a twisted $S$-module structure $\Lambda=(\alpha,\lambda,f)$ on $A$ by the formulas (see~\cite{Lausch}):
	\begin{align}
	\alpha&=i\m\circ\rho|_{E(S)},\label{alpha-from-rho}\\ 
	\lambda_s(a)&=i\m(\rho(s)i(a)\rho(s)\m),\label{lambda-from-rho}\\
	f(s,t)&=i\m(\rho(s)\rho(t)\rho(st)\m).\label{f-from-rho}
	\end{align}
	In this case $A*_\Lambda S$ is equivalent to $U$, the map $a\delta_s\mapsto i(a)\rho(s)$ being the corresponding isomorphism (see~\cite[Theorem 9.1]{Lausch}).
	
	A transversal $\rho$ of $j$ is said to be {\it order-preserving}~\cite{Sieben98,DK2}, when $s\le t\impl\rho(s)\le\rho(t)$ for all $s,t\in S$. It was shown in~\cite{DK2} that $\rho$ is order-preserving if and only if, instead of \cref{f(se_e)-and-f(e_es)}, $\Lambda$ satisfies a stronger condition
	\begin{enumerate}
		\item[\mylabel{Sieben's-condition}{(iv')}] $f(s,e)=\alpha(ses\m)$ and $f(e,s)=\alpha(ess\m)$ for all $s\in S$ and $e\in E(S)$.
	\end{enumerate}
	Inspired by~\cite[Definition 5.4]{Buss-Exel11}, such twisted $S$-modules were called {\it Sieben's} twisted $S$-modules in~\cite{DK2}.
	
	\subsection{The relation between Sieben's twisted modules and twisted partial actions}\label{subsec-Theta-and-Lambda}
	Let $S$ be an $E$-unitary inverse semigroup and $A$ a semilattice of groups. It was proved in~\cite{DK2} that with each Sieben's twisted $S$-module structure $\Lambda=(\alpha,\lambda,f)$ on $A$ one can associate a twisted partial action of $\cG S$ on $A$ as follows:
	\begin{align}
	\cD_x&=\bigsqcup_{s\in x} A_{\alpha(ss\m)},\ \ x\in\cG S,\label{D-from-Lambda}\\
	\0_x(a)&=\lambda_s(a),\mbox{ where }s\in x\mbox{ such that }\alpha(s\m s)=aa\m,\label{0-from-Lambda}\\
	w_{x,y}a&=f(s,s\m t)a,\ \ aw_{x,y}=af(s,s\m t),\label{w-from-Lambda}
	\end{align}
	where $s\in x$ and $t\in xy$, such that $\alpha(ss\m)=\alpha(tt\m)=aa\m$. 
	
	Conversely, given a twisted partial action $\Theta=(\0,w)$ of a group $G$ on a semilattice of groups $A$, there exists~\cite{DK2} an $E$-unitary inverse semigroup $S$, an epimorphism $\kappa:S\to G$ with $\ker\kappa=\sigma$ and an isomorphism $\alpha: E(S)\to E(A)$, such that $\Lambda=(\alpha,\lambda,f)$ is a Sieben's twisted $S$-module structure on $A$, where
	\begin{align}
	\lambda_s(a)&=\0_{\kappa(s)}(\alpha(s\m s)a),\label{lambda-from-Theta}\\
	f(s,t)&=\alpha(stt\m s\m)w_{\kappa(s),\kappa(t)}.\label{f-from-Theta}
	\end{align}
	Notice that one can take $S=E(A)*_\0 G$ with
	\begin{align}
	\kappa(e\delta_x)&=x,\label{kappa(e-delta_x)}\\
	\alpha(e\delta_1)&=e.\label{alpha(e-delta_1)}
	\end{align}

	Up to identification of isomorphic groups and semigroups, this defines a one-to-one correspondence between twisted partial actions of groups on $A$ and Sieben's twisted module structures over $E$-unitary inverse semigroups on $A$. Moreover, equivalent twisted partial actions correspond to equivalent twisted modules (see~\cite[Theorem 9.3]{DK2}).
	
	\subsection{Extensions of semilattices of groups by groups}\label{subsec-ext-of-A-by-G}
	Recall from~\cite{DK2} that an {\it extension} of a semilattice of groups $A$ by a group $G$ is an inverse semigroup $U$ with a monomorphism $i:A\to U$ and an epimorphism $j:U\to G$, such that $i(A)=j\m(1)$.
	
	Two extensions $A\overset{i}{\to}U\overset{j}{\to}G$ and $A\overset{i'}{\to}U'\overset{j'}{\to}G$ of $A$ by $G$ are called {\it equivalent} if there is an isomorphism $\mu:U\to U'$ making the following diagram
	\begin{align}\label{eq-comm-diag-equiv-ext-A-G}
	\begin{tikzpicture}[node distance=1.5cm, auto]
	\node (A) {$A$};
	\node (U) [right of=A] {$U$};
	\node (G) [right of=U] {$G$};
	\node (A') [below of=A]{$A$};
	\node (U') [below of=U] {$U'$};
	\node (G') [below of=S] {$G$};
	\draw[->] (A) to node {$i$} (U);
	\draw[->] (U) to node {$j$} (G);
	\draw[->] (A') to node {$i'$} (U');
	\draw[->] (U') to node {$j'$} (G');
	\draw[-,double distance=2pt] (A) to node {} (A');
	\draw[->] (U) to node {$\mu$} (U');
	\draw[-,double distance=2pt] (G) to node {} (G');
	\end{tikzpicture}
	\end{align}
	commute.
	
	It was proved in~\cite{DK2} that for any extension $A\overset{i}{\to}U\overset{j}{\to}G$ there exists a {\it refinement} $A\overset{i}{\to}U\overset{\pi}{\to}S\overset{\kappa}{\to}G$, where $S$ is an $E$-unitary inverse semigroup, $\pi$ and $\kappa$ are epimorphisms, such that 
	\begin{enumerate}
		\item $A\overset{i}{\to}U\overset{\pi}{\to}S$ is an extension of $A$ by $S$;\label{A->^i-U->^pi-S}
		\item $j=\kappa\circ\pi$.\label{j=kappa-circ-pi}
	\end{enumerate}
	Moreover, it follows that $\ker\kappa=\sigma$.
	
	If $A\overset{i'}{\to}U'\overset{j'}{\to}G$ is an other extension with a refinement $A\overset{i'}{\to}U'\overset{\pi'}{\to}S'\overset{\kappa'}{\to}G$, then any homomorphism $\mu:U\to U'$ making the diagram \cref{eq-comm-diag-equiv-ext-A-G} commute induces a homomorphism $\nu:S\to S'$, such that
	\begin{align}\label{refined-diag-ext-A-G}
	\begin{tikzpicture}[node distance=1.5cm, auto]
	\node (A) {$A$};
	\node (U) [right of=A] {$U$};
	\node (S) [right of=U] {$S$};
	\node (G) [right of=S] {$G$};
	\node (A') [below of=A]{$A$};
	\node (U') [below of=U] {$U'$};
	\node (S') [below of=S] {$S'$};
	\node (G') [below of=G] {$G$};
	\draw[->] (A) to node {$i$} (U);
	\draw[->] (U) to node {$\pi$} (S);
	\draw[->] (S) to node {$\kappa$} (G);
	\draw[->] (A') to node {$i'$} (U');
	\draw[->] (U') to node {$\pi'$} (S');
	\draw[->] (S') to node {$\kappa'$} (G');
	\draw[-,double distance=2pt] (A) to node {} (A');
	\draw[->] (U) to node {$\mu$} (U');
	\draw[->,dashed] (S) to node {$\nu$} (S');
	\draw[-,double distance=2pt] (G) to node {} (G');
	\end{tikzpicture}
	\end{align}
	commutes. Moreover, if $\mu$ is injective, then $\nu$ is injective; if $\mu$ is surjective, then $\nu$ is surjective. In particular, this shows that a refinement is unique up to an isomorphism.
	
	An extension $A\overset{i}{\to}U\overset{j}{\to}G$ is called {\it admissible}~\cite{DK2}, if the corresponding $\pi:U\to S$ has an order-preserving transversal. Up to equivalence, the admissible extensions of $A$ by $G$ are precisely the crossed products $A*_\Theta G$ by twisted partial actions $\Theta$ of $G$ on $A$ (see~\cite{DK2}).
	
	\subsection{Cohomology of inverse semigroups}\label{sec-H^n-of-S}
	Let $A$ be an $S$-module. Following~\cite[p. 275]{Lausch}, when $\alpha$ need not be specified, we shall often use $E(S)$ as an indexing semilattice for the group components of $A$. More precisely, for arbitrary $e\in E(S)$, $A_e$ will mean $\{a\in A\mid aa\m=\alpha(e)\}$. It follows from $\varphi\circ\alpha=\alpha'$ that any morphism $\varphi:A\to A'$ in $\Mod S$ maps $A_e$ to a subset of $A'_e$.
	
	More generally, given a semilattice $L$, an {\it $L$-set} is a disjoint union $T=\bigsqcup_{l\in L}T_l$ of sets $T_l$, $l\in L$. An {\it $L$-map} is a function $f:T\to T'$, such that $f(T_l)\subseteq T_l'$ for all $l\in L$. Thus, any $S$-module is an $E(S)$-set, and any morphism of $S$-modules is an $E(S)$-map.
	
	Now, considering the forgetful functor from $\Mod S$ to the category of $E(S)$-sets, one can naturally define the free $S$-module $F(T)$ over an $E(S)$-set $T$. It turns out (see~\cite[Proposition 3.1]{Lausch}) that such a module exists and can be constructed in the following way (we use additive notation). For any $e\in E(S)$ the component $F(T)_e$ is the free abelian group over the set of pairs (written as formal products) $\{st\in S\times T\mid ss\m=e,\ \ t\in\bigcup_{f\ge s\m s}T_f\}$. The sum of $st\in F(T)_e$ and $s't'\in F(T)_{e'}$ is the formal sum $(e's)t+(es')t'$ in $F(T)_{ee'}$. Clearly, $\alpha(e)=0_e$, where $0_e$ is the zero of $F(T)_e$. The endomorphism $\lambda_s$ of $F(T)$ is defined on a generator $s't'$ by $\lambda_s(s't')=(ss')t'$. An element $t$ of $T_e$, $e\in E(S)$, is identified with $et\in F(T)_e$, determining an embedding of $T$ into $F(T)$ in the category of $E(S)$-sets. 
	
	\begin{rem}\label{rem-lambda_s(t)=st}
		One should be careful when identifying the elements of $T$ with their images in $F(T)$. For instance, taking $t\in T_f$ and considering it as an element of $F(T)$, one could expect that $\lambda_s(t)=st$. However, $st$ need not belong to $F(T)$. In fact, $\lambda_s(t)=\lambda_s(ft)=(sf)t$. Nevertheless, if $st\in F(T)$, i.\,e. $s\m s\le f$, then 
		\begin{align}\label{eq-sf=s}
		sf=(ss\m s)f=s(s\m sf)=s(s\m s)=s, 
		\end{align}
		so $\lambda_s(t)=st$.
	\end{rem}
	
	It follows that $F(T)$ is projective in $\Mod S$ and that for any $A\in\Mod S$ there is an epimorphism $F(A)\to A$, so $\Mod S$ has enough projectives. Considering the semilattice $\bb Z_S$ of copies $(\bb Z_S)_e=\{n_e\mid n\in\bb Z\}$ of $\bb Z$ with $n_e+m_f=(n+m)_{ef}$ as a ``trivial'' $S$-module in the sense that $\lambda_s(n_e)=n_{ses\m}$ and $\alpha(e)=0_e$, Lausch defines in~\cite[Section 5]{Lausch} the cohomology groups of $S$ with values in $A$ as $H^n(S,A)=R^n\Hom(-,A)$ applied\footnote{Notice that Lausch uses the general notion ``cohomology functor'' defined axiomatically, but what he constructs is exactly the right derived functor of $\Hom(-,A)$.} to $\bb Z_S$. 
	
	\section{On the cohomology of inverse monoids}\label{sec-H^n(S_A)}
	
	\subsection{The free resolutions \texorpdfstring{$C(S)$}{C(S)} and \texorpdfstring{$D(S)$}{D(S)} of \texorpdfstring{$\bb Z_S$}{ZS}}\label{sec-res-C-and-D}
	Let $S$ be an inverse monoid. Recall from~\cite[pp. 280--282]{Lausch} that in this case the cohomology groups of $S$ can be computed using the free resolutions $C(S)$ and $D(S)$ of $\bb Z_S$ in $\Mod S$, where $C_n(S)=F(V_n(S))$ and $D_n(S)=F(W_n(S))$, $n\ge 0$, with
	\begin{align}
	V_0(S)_e&=\begin{cases}
	\{(\phantom{x})\}, & e=1_S,\\
	\emptyset, & e\ne 1_S;
	\end{cases}\label{V_0(S)_e}\\
	V_n(S)_e&=\{(s_1,\dots,s_n)\mid s_1\dots s_ns\m_n\dots s\m_1=e\},\ \ n\ge 1;\label{V_n(S)_e}\\
	W_0(S)&=V_0(S);\label{W_0(S)=V_0(S)}\\
	W_n(S)_e&=\{(s_1,\dots,s_n)\in V_n(S)_e\mid s_i\ne 1_S,\ \ i=1,\dots,n\},\ \ n\ge 1.\label{W_n(S)_e}
	\end{align}
	
	The $S$-morphisms $\partial_0':C_0(S)\to\bb Z_S$ and $\partial_n':C_n(S)\to C_{n-1}(S)$, $n\ge 1$, are defined in the following way:
	\begin{align}
	\partial_0'(\phantom{x})&=1_{1_S};\label{eq-epsilon'}\\
	\partial'_1(s)&=s(\phantom{x})-(\phantom{x});\label{partial_1'}\\
	\partial'_n(s_1,\dots,s_n)&=\lambda_{s_1}(s_2,\dots,s_n)\notag\\
	&\quad+\sum_{i=1}^{n-1}(-1)^i(s_1,\dots,s_is_{i+1},\dots,s_n)\notag\\
	&\quad+(-1)^n(s_1,\dots,s_{n-1}),\ \ n\ge 2.\label{partial_n'}
	\end{align}
	
	The morphisms $\partial_0'':D_0(S)\to \bb Z_S$ and $\partial''_1:D_1(S)\to D_0(S)$ are simply $\partial_0'$ and $\partial'_1|_{D_1(S)}$, respectively. To define $\partial''_n:D_n(S)\to D_{n-1}(S)$, $n\ge 2$, we use \cref{partial_n'} with the following modification. If $n\ge 2$ and the term $u(v_1,\dots,v_{n-1})\in C_{n-1}(S)$ appears (with some sign) on the right-hand side of \cref{partial_n'}, then it is replaced by $0_{uu\m}$ whenever $1_S\in\{v_1,\dots,v_{n-1}\}$. For instance, the summand $(-1)^i(s_1,\dots,s_is_{i+1},\dots,s_n)$, which, as we know, is identified with 
	$$
	(-1)^ie(s_1,\dots,s_is_{i+1},\dots,s_n)\in C_{n-1}(S)_e,\ e=s_1\dots s_ns\m_n\dots s\m_1,
	$$
	is set to be $0_e$, if $s_is_{i+1}=1_S$. 
	
	\begin{rem}\label{rem-Lausch's-partial''}
		In~\cite[p. 282]{Lausch} Lausch required additionally that all the elements of $C_{n-1}(S)$ of the form $1_S(v_1,\dots,v_{n-1})$ be identified with $0_{1_S}$. Assuming this, we would get for $s,t\in S$ with $s,t,st\ne 1_S$ and $ss\m=tt\m=1_S$ that
		\begin{align*}
		\partial''_1\circ\partial''_2(s,t)&=\partial''_1(\lambda_s(t)-(st)+(s))\\
		&=\partial''_1(stt\m(t)-stt\m s\m(st)+ss\m(s))\\
		&=\partial''_1(s(t))=s(t(\phantom{x})-(\phantom{x}))=st(\phantom{x})-s(\phantom{x}),
		\end{align*}
		which does not equal $0_{stt\m s\m}$, if $st\ne s$ (for example, when $S$ is a group, the latter follows from $t\ne 1_S$, so such $s$ and $t$ can be easily found in, say, $S=\bb Z_3$). Thus, $\partial''_1\circ\partial''_2$ is not necessarily zero, demonstrating that Lausch's definition should be corrected.
	\end{rem}
	
	Observe that, when $S$ is a group, $C(S)$ coincides with the standard resolution of $\bb Z$ in the bar notation (or, shortly, the bar resolution). Then $D(S)$ is the normalized bar resolution~\cite[I.5]{Brown}. One should also note that Maclane~\cite[IV.5]{Maclane} uses the term ``bar resolution'' for the normalized bar resolution.
	
	For the exactness of the sequence $C(S)$ Lausch implicitly uses in~\cite[p. 281]{Lausch} the Maclane's argument~\cite[p. 115]{Maclane}\footnote{Lausch refers to Maclane's book on p. 280 while proving the exactness of some other sequence, on p. 281 he uses the same argument without any reference.} by constructing $E(S)$-maps $\sigma'_{-1}:\bb Z_S\to C_0(S)$ and $\sigma'_n:C_n(S)\to C_{n+1}(S)$, $n\ge 0$, such that 
	\begin{enumerate}
		\item for each $e\in E(S)$ the restrictions $\sigma'_{-1}|_{(\bb Z_S)_e}$ and $\sigma'_n|_{C_n(S)_e}$, $n\ge 0$, are morphisms of abelian groups $(\bb Z_S)_e\to C_0(S)_e$ and $C_n(S)_e\to C_{n+1}(S)_e$, respectively;\label{sigma'_n|C_n(S)_e}
		\item the following equalities are fulfilled:\label{sigma'_n-homotopy}
		\begin{align}
		\partial'_0\circ\sigma'_{-1}&=
		\id_{\bb Z_S};\label{eq-epsilon'tau'=id}\\
		\partial'_{n+1}\circ\sigma'_n+\sigma'_{n-1}\circ\partial'_n&=
		\id_{C_n(S)},\ \ n\ge 0 \label{eq-partial'_(n+1)sigma'_n+sigma'_(n-1)partial'_n=id}.		
		\end{align}
	\end{enumerate}
	
	\begin{lem}\label{lem-ker(partial'_n)-in-im(partial'_(n+1))}
		If the maps $\sigma_n'$, $n\ge -1$, satisfying \cref{sigma'_n|C_n(S)_e,sigma'_n-homotopy} exist, then $\partial_0'$ is surjective and $\ker{\partial'_n}\subseteq\im{\partial'_{n+1}}$, $n\ge 0$.
	\end{lem}
	\begin{proof}
		Surjectivity of $\partial_0'$ is explained by \cref{eq-epsilon'tau'=id}. If $c\in\ker{\partial_n'}$, i.\,e. $c\in C_n(S)_e$ and $\partial_n'(c)=0_e$ for some $e\in E(S)$, then $\sigma_{n-1}'\circ\partial_n'(c)=\sigma_{n-1}'(0_e)$ which is zero of $C_n(S)_e$ in view of \cref{sigma'_n|C_n(S)_e}. So, $c=\partial'_{n+1}\circ\sigma'_n(c)\in\im{\partial'_{n+1}}$ by \cref{eq-partial'_(n+1)sigma'_n+sigma'_(n-1)partial'_n=id}.
	\end{proof}
	
	\begin{rem}\label{rem-<sigma'n(C_n(S))>-ne-C-(n+1)(S)}
		According to~\cite[p. 115]{Maclane} in order to prove the converse inclusions, i.\,e. that $C(S)$ is a chain complex in $\Mod S$, one should somehow deduce from $\partial'_n\circ\partial'_{n+1}\circ\sigma'_n=0$ that $\partial'_n\circ\partial'_{n+1}=0$, $n\ge 0$. In the classical case Maclane uses the fact that (in our notations) $\sigma'_n(C_n(S))$ generates $C_{n+1}(S)$ as an $S$-module, which is not true for $\sigma'_n$ introduced by Lausch in~\cite[p. 281]{Lausch} (see \cref{rem-sigma'_n(C_n(S))-ne-C_(n+1)(S)}).
	\end{rem}
	
	Nevertheless, we may still reduce the problem to the classical formula from the group cohomology. 
	
	\begin{lem}\label{lem-partial'^2=0}
		For each $n\ge 0$ we have
		\begin{align}\label{eq-partial'^2=0}
		\partial'_n\circ\partial'_{n+1}=0.
		\end{align}
	\end{lem}
	\begin{proof}
		When $n\ge 1$, the formulas \cref{partial_1',partial_n'} have exactly the same form as the ones for the bar resolution~\cite[IV.(5.3)]{Maclane} (we should only identify $(\phantom{x})$ with $[\phantom{x}]$, $(s_1,\dots,s_n)$ with $[s_1|\dots|s_n]$ and the application of $\lambda_s$ with the multiplication by $s$ on the left). Expanding the equality $\partial_n\circ\partial_{n+1}[s_1|\dots|s_{n+1}]=0$, $n\ge 1$, written for the bar resolution, and identifying each summand with an element of $C_{n-1}(S)$ as explained above, we obtain a formal proof that $\partial'_n\circ\partial'_{n+1}(s_1,\dots,s_{n+1})=0_{s_1\dots s_{n+1}s\m_{n+1}\dots s\m_1}$ with the only difference that $0$ (when appears as a sum of two terms with opposite signs) should be replaced by the zero of the corresponding component. For example, $\partial_1\circ\partial_2[s|t]=0$ expands to
		\begin{align*}
		\partial_1\circ\partial_2[s|t]&=\partial_1(s[t]-[st]+[s])\\
		&=s(t[\phantom{x}]-[\phantom{x}])-(st[\phantom{x}]-[\phantom{x}])+(s[\phantom{x}]-[\phantom{x}])\\
		&=(st[\phantom{x}]-st[\phantom{x}])+(s[\phantom{x}]-s[\phantom{x}])+([\phantom{x}]-[\phantom{x}])\\
		&=0.
		\end{align*}
		This gives
		\begin{align*}
		0_{stt\m s\m}&=0_{stt\m s\m}+0_{ss\m}+0_{1_S}\\
		&=(\lambda_{st}(\phantom{x})-\lambda_{st}(\phantom{x}))
		+(\lambda_s(\phantom{x})-\lambda_s(\phantom{x}))
		+((\phantom{x})-(\phantom{x}))\\
		&=\lambda_s(\lambda_t(\phantom{x})-(\phantom{x}))
		-(\lambda_{st}(\phantom{x})-(\phantom{x}))
		+(\lambda_s(\phantom{x})-(\phantom{x}))\\
		&=\partial'_1(\lambda_s(t)-(st)+(s))\\
		&=\partial'_1\circ\partial'_2(s,t).
		\end{align*}
		
		As to $\partial'_0\circ\partial'_1$, it will be calculated by \cref{eq-epsilon',partial_1'} explicitly:
		$$
		\partial'_0\circ\partial'_1(s)=
		\lambda_s(1_{1_S})-1_{1_S}=1_{ss\m}-1_{1_S}=(1-1)_{ss\m\cdot 1_S}=0_{ss\m}.
		$$
	\end{proof}
	
	Now, following~\cite[p. 281]{Lausch}, we define the maps $\sigma'_{-1}$ and $\sigma'_n$, $n\ge 0$, on the generators of the group components of $\bb Z_S$ and $C_n(S)$, $n\ge 0$, respectively, by
	\begin{align}
	\sigma'_{-1}(1_e)&=e(\phantom{x});\label{eq-tau'}\\
	\sigma'_0(s(\phantom{x}))&=(s);\label{eq-sigma'_0}\\
	\sigma'_n(s(s_1,\dots,s_n))&=(s,s_1,\dots,s_n).\label{eq-sigma'_n}
	\end{align}
	Obviously, $\sigma'_{-1}(1_e)\in C_0(S)_e$ and $\sigma'_0(s(\phantom{x}))\in C_1(S)_{ss\m}$. Now observe by \cref{eq-sf=s} that for $s(s_1,\dots,s_n)\in C_n(S)=F(V_n(S))$ we have $sfs\m=ss\m$, where $f=s_1\dots s_ns\m_n\dots s\m_1$.
	Hence $\sigma'_n(s(s_1,\dots,s_n))\in C_{n+1}(S)_{ss\m}$, $n\ge 1$. Thus, the maps $\sigma'_n$, $n\ge -1$, respect the partitions of $\bb Z_S$ and $C_n(S)$, $n\ge 0$, into components, so they uniquely extend to the $E(S)$-maps $\bb Z_S\to C_0(S)$ and $C_n(S)\to C_{n+1}(S)$, $n\ge 0$, with property \cref{sigma'_n|C_n(S)_e} above. We prove that \cref{sigma'_n-homotopy} is also fulfilled.
	
	\begin{lem}\label{lem-tau'-and-sigma'-satisfy-(ii)}
		The functions $\sigma'_n$, $n\ge -1$, satisfy the equalities \cref{eq-epsilon'tau'=id,eq-partial'_(n+1)sigma'_n+sigma'_(n-1)partial'_n=id}.
	\end{lem}
	\begin{proof}
		Taking the generator $1_e$ of $(\bb Z_S)_e$, we easily verify by \cref{eq-epsilon',eq-tau'} that
		$$
		\partial_0'\circ\sigma_{-1}'(1_e)=\partial_0'(e(\phantom{x}))=\lambda_e(\partial_0'(\phantom{x}))=\lambda_e(1_{1_S})=1_e,
		$$
		which is \cref{eq-epsilon'tau'=id}. 
		
		Furthermore, for any $s\in S$ by \cref{eq-epsilon',partial_1',eq-tau',eq-sigma'_0}:
		\begin{align*}
		\partial'_1\circ\sigma'_0(s(\phantom{x}))+\sigma_{-1}'\circ\partial_0'(s(\phantom{x}))
		&=\partial'_1(s)+\sigma_{-1}'\circ\lambda_s(1_{1_S})\\
		&=s(\phantom{x})-(\phantom{x})+\sigma_{-1}'(1_{ss\m})\\
		&=s(\phantom{x})-(\phantom{x})+ss\m(\phantom{x})\\
		&=s(\phantom{x})-(\phantom{x})+0_{ss\m}+(\phantom{x})\\
		&=s(\phantom{x}),
		\end{align*}
		giving \cref{eq-partial'_(n+1)sigma'_n+sigma'_(n-1)partial'_n=id} for $n=0$. 
		
		As to \cref{eq-partial'_(n+1)sigma'_n+sigma'_(n-1)partial'_n=id} for $n\ge 1$, using \cref{partial_n',eq-sigma'_n} we see that
		\begin{align}
		\partial'_{n+1}\circ\sigma'_n(s(s_1,\dots,s_n))&=\partial'_{n+1}(s,s_1,\dots,s_n)\notag\\
		&=\lambda_s(s_1,\dots,s_n)-(ss_1,s_2,\dots,s_n)\notag\\
		&\quad+\sum_{i=1}^{n-1}(-1)^{i+1}(s,s_1,\dots,s_is_{i+1},\dots,s_n)\notag\\
		&\quad+(-1)^{n+1}(s,s_1,\dots,s_{n-1})\label{eq-partial_(n+1)sigma'_n}
		\end{align}
		and
		\begin{align*}
		\sigma'_{n-1}\circ\partial'_n(s(s_1,\dots,s_n))&=\sigma'_{n-1}(\lambda_s(\partial'_n(s_1,\dots,s_n)))\\
		&=\sigma'_{n-1}(\lambda_s(\lambda_{s_1}(s_2,\dots,s_n)\\
		&\quad+\sum_{i=1}^{n-1}(-1)^i(s_1,\dots,s_is_{i+1},\dots,s_n)\\
		&\quad+(-1)^n(s_1,\dots,s_{n-1}))).
		\end{align*}
		
		We first observe that $\lambda_{s_1}(s_2,\dots,s_n)$ and $(s_1,\dots,s_is_{i+1},\dots,s_n)$, $1\le i\le n-1$, are in the component $C_{n-1}(S)_f$ of $C_{n-1}(S)$, while $(s_1,\dots,s_{n-1})\in C_{n-1}(S)_{f'}$, where $f=s_1\dots s_ns\m_n\dots s\m_1$ and $f'=s_1\dots s_{n-1}s\m_{n-1}\dots s\m_1$. Since clearly $f\le f'$, then using the formula for the addition in a free  module, we replace $(s_1,\dots,s_{n-1})$ by $f(s_1,\dots,s_{n-1})\in C_{n-1}(S)_f$. Now, applying $\lambda_s$, we get
		\begin{align}
		\sigma'_{n-1}\circ\partial'_n(s(s_1,\dots,s_n))&=\sigma'_{n-1}(\lambda_{ss_1}(s_2,\dots,s_n)\notag\\
		&\quad+\sum_{i=1}^{n-1}(-1)^i\lambda_s(s_1,\dots,s_is_{i+1},\dots,s_n)\notag\\
		&\quad+(-1)^nsf(s_1,\dots,s_{n-1})),\label{eq-sigma'_(n-1)partial'_n}
		\end{align}
		all the summands being in the same component of $C_{n-1}(S)$.
		
		Thanks to \cref{eq-sf=s} the product $sf$ equals $s$, because $s(s_1,\dots,s_n)\in F(V_n(S))$ with $(s_1,\dots,s_n)\in V_n(S)_f$. Further, we would like to rewrite $\lambda_{ss_1}$ and $\lambda_s$ in \cref{eq-sigma'_(n-1)partial'_n} as the multiplications on the left by $ss_1$ and $s$, respectively. By \cref{rem-lambda_s(t)=st} we need to check that after doing this we shall obtain elements of $C_{n-1}(S)=F(V_{n-1}(S))$. The fact that $s(s_1,\dots,s_is_{i+1},\dots,s_n)\in C_{n-1}(S)$, $1\le i\le n-1$, follows from $s(s_1,\dots,s_n)\in C_n(S)$, because $(s_1,\dots,s_is_{i+1},\dots,s_n)$ and $(s_1,\dots,s_n)$ belong to the components of $V_{n-1}(S)$ and $V_n(S)$, respectively, with the same index $f$. To prove that $ss_1(s_2,\dots,s_n)\in C_{n-1}(S)$, we make sure that $s\m_1 es_1\le e'$, where $e=s\m s$ and $e'=s_2\dots s_ns\m_n\dots s\m_2$. We know from $s(s_1,\dots,s_n)\in C_n(S)$ that $e\le s_1e's\m_1$. Then $s\m_1 es_1\le s\m_1 s_1e'$ and hence
		$$
		s\m_1 es_1\cdot e'=s\m_1 e(s_1s\m_1 s_1)e'=s\m_1 es_1\cdot s\m_1 s_1e'=s\m_1 es_1,
		$$
		as desired. 
		
		Thus, due to the fact that $\sigma'_{n-1}$ is additive on each component of $C_{n-1}(S)$, we have
		\begin{align}
		\sigma'_{n-1}\circ\partial'_n(s(s_1,\dots,s_n))&=(ss_1,s_2,\dots,s_n)\notag\\
		&\quad+\sum_{i=1}^{n-1}(-1)^i(s,s_1,\dots,s_is_{i+1},\dots,s_n)\notag\\
		&\quad+(-1)^n(s,s_1,\dots,s_{n-1}).\label{eq-sigma'_(n-1)partial'_n-final}
		\end{align}
		Adding \cref{eq-partial_(n+1)sigma'_n} and \cref{eq-sigma'_(n-1)partial'_n-final} we obtain $\lambda_s(s_1,\dots,s_n)$, which is $s(s_1,\dots,s_n)$ thanks to \cref{rem-lambda_s(t)=st}.
	\end{proof}
	
	\begin{prop}\label{C(S)-is-exact}
		The sequence $C(S)$ is exact.
	\end{prop}
	\begin{proof}
		This follows from \cref{lem-ker(partial'_n)-in-im(partial'_(n+1)),lem-partial'^2=0,lem-tau'-and-sigma'-satisfy-(ii)}.
	\end{proof}
	
	\begin{rem}\label{rem-sigma'_n(C_n(S))-ne-C_(n+1)(S)}
		For any $n\ge 1$ the image $\sigma'_n(C_n(S))$ consists of those generators $(s_1,\dots,s_{n+1})$ of $C_{n+1}(S)$, for which $s\m_1 s_1\le s_2\dots s_{n+1}s\m_{n+1}\dots s\m_2$. For example, when $S$ is obtained from an inverse semigroup by adjoining identity $1_S$, the only $(n+1)$-tuple from $\sigma'_n(C_n(S))$ with $s_1=1_S$ is $(1_S,\dots,1_S)$. Hence, $\sigma'_n(C_n(S))$ generates a proper submodule of $C_{n+1}(S)$.
	\end{rem}
	
	We obtain the exactness of $D(S)$ as a consequence of the exactness of $C(S)$. To this end we introduce the epimorphisms of $S$-modules $\zeta_{-1}:\bb Z_S\to\bb Z_S$ and $\zeta_n:C_n(S)\to D_n(S)$, $n\ge 0$, by $\zeta_{-1}=\id_{\bb Z_S}$, $\zeta_0=\id_{C_0(S)}$ and
	\begin{align}\label{eq-zeta_n}
	\zeta_n(s_1,\dots,s_n)=
	\begin{cases}
	(s_1,\dots,s_n),                  & (s_1,\dots,s_n)\in W_n(S),\\
	0_{s_1\dots s_ns\m_n\dots s\m_1}, & (s_1,\dots,s_n)\not\in W_n(S)
	\end{cases}
	\end{align}
	for $n\ge 1$ and $(s_1,\dots,s_n)\in V_n(S)$. It follows that $\zeta_n$ is identity on $D_n(S)\subseteq C_n(S)$, $n\ge 0$.
	
	\begin{lem}\label{lem-zeta-morph-of-compl}
		For any $n\ge 0$ we have $\partial''_n\circ\zeta_n=\zeta_{n-1}\circ\partial'_n$.
	\end{lem}
	\begin{proof}
		The case $n=0$ is trivial: $\partial''_0\circ\zeta_0=\partial''_0=\partial'_0=\zeta_{-1}\circ\partial'_0$. 
		
		For $n=1$ we need to show that $\partial''_1\circ\zeta_1=\zeta_0\circ\partial'_1=\partial'_1$. Take an arbitrary generator $(s)\in V_1(S)$. If $s\ne 1_S$, then $(s)\in W_1(S)$, so $\partial''_1\circ\zeta_1(s)=\partial''_1(s)=\partial'_1(s)$ by the definitions of $\zeta_1$ and $\partial''_1$. If $s=1_S$, i.\,e. $(s)\not\in W_1(S)$, then $\partial''_1\circ\zeta_1(s)=\partial''_1(0_{1_S})=0_{1_S}$, and $\partial'_1(s)=1_S(\phantom{x})-(\phantom{x})=1_S(\phantom{x})-1_S(\phantom{x})=0_{1_S}$.
		
		If $n\ge 2$ and $(s_1,\dots,s_n)\in W_n(S)$, then one needs to prove that
		\begin{align*}
		\partial''_n(s_1,\dots,s_n)=\zeta_{n-1}\circ\partial'_n(s_1,\dots,s_n).
		\end{align*}
		It is simply the definition of $\partial''_n$ rewritten in terms of $\partial'_n$ and $\zeta_{n-1}$.
		
		If $n\ge 2$ and $(s_1,\dots,s_n)\not\in W_n(S)$, then the desired equality reduces to
		\begin{align*}
		\zeta_{n-1}\circ\partial'_n(s_1,\dots,s_n)=0_{s_1\dots s_ns\m_n\dots s\m_1}.
		\end{align*}
		Consider three possible cases.
		
		(a) $s_1=1_S$. Then
		\begin{align*}
		\zeta_{n-1}\circ\partial'_n(s_1,\dots,s_n)&=\zeta_{n-1}(\lambda_{1_S}(s_2,\dots,s_n)-(1_S\cdot s_2,\dots,s_n))\\
		&\quad+\sum_{i=2}^{n-1}(-1)^i\zeta_{n-1}(1_S,s_2,\dots,s_is_{i+1},\dots,s_n)\\
		&\quad+(-1)^n\zeta_{n-1}(1_S,s_2,\dots,s_{n-1}).
		\end{align*}
		The difference $\lambda_{1_S}(s_2,\dots,s_n)-(1_S\cdot s_2,s_3,\dots,s_n)$ is $0_{s_2\dots s_ns\m_n\dots s\m_2}$, and it is mapped by $\zeta_{n-1}$ to itself as an element of $D_n(S)$. All the other terms under the sign of $\zeta_{n-1}$ are also mapped to zeros of the corresponding components of $D_n(S)$, since they contain $1_S$. Thus, the sum is $0_{s_2\dots s_ns\m_n\dots s\m_2}=0_{s_1\dots s_ns\m_n\dots s\m_1}$.
		
		(b) $s_i=1_S$ for some $2\le i\le n-1$. In this case
		\begin{align}
		\zeta_{n-1}\circ\partial'_n(s_1,\dots,s_n)&=\zeta_{n-1}(\lambda_{s_1}(s_2,\dots,1_S,\dots,s_n))\notag\\
		&\quad+\sum_{j=1}^{i-2}(-1)^j\zeta_{n-1}(s_1,\dots,s_js_{j+1},\dots,1_S,\dots,s_n)\notag\\
		&\quad+(-1)^{i-1}\zeta_{n-1}(s_1,\dots,s_{i-2},s_{i-1}\cdot 1_S,s_{i+1},\dots,s_n)\label{eq-zeta(s_1...s_(i-1)1...s_n)}\\
		&\quad+(-1)^i\zeta_{n-1}(s_1,\dots,s_{i-1},1_S\cdot s_{i+1},s_{i+2},\dots,s_n)\label{eq-zeta(s_1...1s_(i+1)...s_n)}\\
		&\quad+\sum_{j=i+1}^{n-1}(-1)^j\zeta_{n-1}(s_1,\dots,1_S,\dots,s_js_{j+1},\dots,s_n)\notag\\
		&\quad+(-1)^n\zeta_{n-1}(s_2,\dots,1_S,\dots,s_{n-1}).\notag
		\end{align}
		The summands \cref{eq-zeta(s_1...s_(i-1)1...s_n)} and \cref{eq-zeta(s_1...1s_(i+1)...s_n)} differ only by the sign, all the other summands are zero by the definition of $\zeta_{n-1}$. 
		
		(c)  $s_n=1_S$. Then
		\begin{align*}
		\zeta_{n-1}\circ\partial'_n(s_1,\dots,s_n)&=\zeta_{n-1}(\lambda_{s_1}(s_2,\dots,1_S))\\
		&\quad+\sum_{i=1}^{n-2}(-1)^i\zeta_{n-1}(s_1,\dots,s_is_{i+1},\dots,1_S)\\
		&\quad+(-1)^{n-1}\zeta_{n-1}((s_1,\dots,s_{n-1}\cdot 1_S)-(s_1,\dots,s_{n-1})),
		\end{align*}
		which is clearly zero.
	\end{proof}
	
	\begin{cor}\label{cor-partial''^2=0}
		For all $n\ge 0$ the composition $\partial''_n\circ\partial''_{n+1}$ is zero.
	\end{cor}
    \begin{proof}
     Indeed, consider the composition $(\partial''_n\circ\partial''_{n+1})\circ\zeta_{n+1}$. Using \cref{lem-zeta-morph-of-compl} twice, we obtain
     \begin{align*}
     \partial''_n\circ(\partial''_{n+1}\circ\zeta_{n+1})=(\partial''_n\circ\zeta_n)\circ\partial'_{n+1}=\zeta_{n-1}\circ(\partial'_n\circ\partial'_{n+1}),
     \end{align*}
     which is zero by \cref{lem-partial'^2=0}. The result now follows from surjectivity of $\zeta_{n+1}$.
    \end{proof}

	We proceed with the construction of the corresponding maps $\sigma''_n$, $n\ge -1$.
	
	\begin{lem}\label{lem-sigma''_n}
		There is a (uniquely defined) collection of functions $\sigma''_{-1}:\bb Z_S\to D_0(S)$, $\sigma''_n:D_n(S)\to D_{n+1}(S)$, $n\ge 0$, satisfying
		\begin{align}\label{eq-sigma''-zeta=zeta-sigma'}
		\sigma''_n\circ\zeta_n=\zeta_{n+1}\circ\sigma'_n
		\end{align}
		for all $n\ge -1$. Moreover, it follows that $\sigma''_n$ is a homomorphism of abelian groups when restricted to a group component of the corresponding module.
	\end{lem}
	\begin{proof}
		Since $\zeta_{-1}=\id_{\bb Z_S}$ and $\zeta_n$ is identity on $D_n(S)$, we immediately obtain that for $n=-1$ the equality \cref{eq-sigma''-zeta=zeta-sigma'} becomes $\sigma''_{-1}=\sigma'_{-1}$ (so that $\sigma''_{-1}|_{(\bb Z_S)_e}$ is automatically a homomorphism as $\sigma'_{-1}|_{(\bb Z_S)_e}$ is), and for $n\ge 0$ it is equivalent on $D_n(S)\subseteq C_n(S)$ to the fact that $\sigma''_n=\zeta_{n+1}\circ\sigma'_n|_{D_n(S)}$. It follows from $D_n(S)_e\subseteq C_n(S)_e$ that $\sigma'_n|_{D_n(S)_e}:D_n(S)_e\to C_{n+1}(S)_e$ is a homomorphism as the restriction of the homomorphism $\sigma'_n|_{C_n(S)_e}$ to $D_n(S)_e$. Furthermore, $\zeta_{n+1}$, being a morphism of $S$-modules, restricts to a homomorphism $\zeta_{n+1}|_{C_{n+1}(S)_e}:C_{n+1}(S)_e\to D_{n+1}(S)_e$. Hence, $\sigma''_n|_{D_n(S)_e}=\zeta_{n+1}|_{C_{n+1}(S)_e}\circ\sigma'_n|_{D_n(S)_e}$ is a homomorphism $D_n(S)_e\to D_{n+1}(S)_e$.
		
		Now we prove that \cref{eq-sigma''-zeta=zeta-sigma'} holds on any generator of $C_n(S)_e$ which does not belong to $D_n(S)_e$, $n\ge 0$, $e\in E(S)$. If $n=0$, then there is nothing to be proved, because $D_0(S)=C_0(S)$. If $n\ge 1$, then $\zeta_n$ is zero on such a generator $s(s_1,\dots,s_n)$, so \cref{eq-sigma''-zeta=zeta-sigma'} transforms into
		\begin{align*}
		\zeta_{n+1}\circ\sigma'_n(s(s_1,\dots,s_n))=0_{ss\m},
		\end{align*}
		when $s_i=1_S$ for some $1\le i\le n$. The latter is a straightforward consequence of \cref{eq-sigma'_n,eq-zeta_n} (see also \cref{eq-sf=s}).
	\end{proof}
	
	\begin{cor}\label{cor-sigma''-contr-homotopy}
		One has
		\begin{align}
		\partial''_0\circ\sigma''_{-1}&=\id_{\bb Z_S};\label{eq-partial''sigma''_(-1)=id}\\
		\partial''_{n+1}\circ\sigma''_n+\sigma''_{n-1}\circ\partial''_n&=\id_{D_n(S)},\ \ n\ge 0 \label{eq-partial''_(n+1)sigma''_n+sigma''_(n-1)partial''_n=id}.		
		\end{align}
	\end{cor}
    \begin{proof}
    Equality \cref{eq-partial''sigma''_(-1)=id} is simply \cref{eq-epsilon'tau'=id}. To establish \cref{eq-partial''_(n+1)sigma''_n+sigma''_(n-1)partial''_n=id}, use \cref{lem-zeta-morph-of-compl,eq-sigma''-zeta=zeta-sigma',eq-partial'_(n+1)sigma'_n+sigma'_(n-1)partial'_n=id} to get
    \begin{align*}
    (\partial''_{n+1}\circ\sigma''_n+\sigma''_{n-1}\circ\partial''_n)\circ\zeta_n&=\partial''_{n+1}\circ(\sigma''_n\circ\zeta_n)+\sigma''_{n-1}\circ(\partial''_n\circ\zeta_n)\\
    &=(\partial''_{n+1}\circ\zeta_{n+1})\circ\sigma'_n+(\sigma''_{n-1}\circ\zeta_{n-1})\circ\partial'_n\\
    &=\zeta_n\circ(\partial'_{n+1}\circ\sigma'_n+\sigma'_{n-1}\circ\partial'_n)\\
    &=\zeta_n.
    \end{align*}
    It remains to ``cancel'' the epimorphism $\zeta_n$.
    \end{proof}

	\begin{prop}\label{D(S)-is-exact}
		The sequence $D(S)$ is exact.
	\end{prop}
	\begin{proof}
		This follows from \cref{cor-partial''^2=0,cor-sigma''-contr-homotopy,lem-ker(partial'_n)-in-im(partial'_(n+1))}.
	\end{proof}
	
	\subsection{A connection between \texorpdfstring{$H^n(S,A)$}{Hn(S,A)} and \texorpdfstring{$H^n(S^1,A^1)$}{Hn(S1,A1)}}
	Let $A$ be an $S$-module. By adjoining identities $1_S$ and $1_A$ to $S$ and $A$ we obtain the $S^1$-module $A^1$ (see~\cite[p. 285]{Lausch}, where one uses the additive notation). In particular, $1_S$ acts on $A^1$ trivially, and $\lambda_s(1_A)=\alpha(s s\m)$.
	
	\begin{prop}\label{H^n(S_A)=H^n(S^1_A^1)}
		Let $S$ be an inverse monoid. For any $S$-module $A$ and for all $n\ge 2$ we have
		$$
		H^n(S,A)\cong H^n(S^1,A^1).
		$$
	\end{prop}
	\begin{proof}
		We see that $W_n(S^1)_e=V_n(S)_e$ for all $n\ge 1$ and $e\in E(S)$. Moreover, $W_n(S^1)_{1_S}=\emptyset$. Hence, $D_n(S^1)_e=C_n(S)_e$ and $D_n(S^1)_{1_S}$ is the trivial group $\{0_{1_S}\}$ for such $n$ and $e$. This implies that there is an isomorphism between $\Hom(C_n(S),A)$ and $\Hom(D_n(S^1),A^1)$ (namely, a morphism $f:C_n(S)\to A$ extends to $\bar f:D_n(S^1)\to A^1$ by $\bar f(0_{1_S})=1_A$). It is clearly an isomorphism of the complexes of abelian groups:
		\begin{align}\label{eq-C(S)-isom-D(S)}
		\begin{tikzpicture}[node distance=4cm, auto]
		\node (C_1) {$\Hom(C_1(S),A)$};
		\node (C_2) [left of=C_1] {$\Hom(C_2(S),A)$};
		\node (dots_C_2) [xshift=1.1cm, left of=C_2] {$\dots$};
		\node (D_1) [yshift=2.5cm, below of=C_1]{$\Hom(D_1(S^1),A^1)$};
		\node (D_2) [left of=D_1] {$\Hom(D_2(S^1),A^1)$};
		\node (dots_D_2) [xshift=1.1cm, left of=D_2] {$\dots$};
		\draw[->] (C_1) to node[above] {$\delta_1^1$} (C_2);
		\draw[->] (C_2) to node[above] {$\delta_1^2$} (dots_C_2);
		\draw[->] (D_1) to node[above] {$\delta_2^1$} (D_2);
		\draw[->] (D_2) to node[above] {$\delta_2^2$} (dots_D_2);
		\draw[-] (C_1) to node {$\wr$} (D_1);
		\draw[-] (C_2) to node {$\wr$} (D_2);
		\end{tikzpicture}
		\end{align}
		where $\delta_1^n(f)=f\circ\partial'_{n+1}$ and $\delta_2^n(g)=g\circ\partial''_{n+1}$ for $n\ge 1$, $f\in\Hom(C_n(S),A)$ and $g\in\Hom(D_n(S^1),A^1)$. 
	\end{proof}
	
	Observe that $\Hom(D_n(S^1),A^1)$, $n\ge 1$, can be identified with the abelian group of functions
	$$
	\{f:S^n\to A\mid f(s_1,\dots,s_n)\in A_{\alpha(s_1\dots s_ns\m_n\dots s\m_1)}\}
	$$
	under the coordinate-wise multiplication, which we denote by $C^n(S^1,A^1)$. Under this identification $\delta_2^n$ becomes a homomorphism, which sends $f\in C^n(S^1,A^1)$ to $\delta_2^nf\in C^{n+1}(S^1,A^1)$, such that
	\begin{align}
	(\delta_2^nf)(s_1,\dots,s_{n+1})&=\lambda_{s_1}(f(s_2,\dots,s_{n+1}))\notag\\
	&\quad\prod_{i=1}^nf(s_1,\dots,s_is_{i+1},\dots,s_{n+1})^{(-1)^i}\notag\\
	&\quad f(s_1,\dots,s_n)^{(-1)^{n+1}}.\label{eq-delta_2^nf}
	\end{align}
	Denote $\ker\delta_2^n$ by $Z^n(S^1,A^1)$ and $\im\delta_2^{n-1}$ by $B^n(S^1,A^1)$, so that the quotient $Z^n(S^1,A^1)/B^n(S^1,A^1)$ is identified with $H^n(S^1,A^1)$, when $n\ge 2$. The elements of $C^n(S^1,A^1)$, $Z^n(S^1,A^1)$ and $B^n(S^1,A^1)$ will be called {\it $n$-cochains}, {\it $n$-cocycles} and {\it $n$-coboundaries}, respectively.
	
	The following proposition completes the result of \cref{H^n(S_A)=H^n(S^1_A^1)}.
	
	\begin{prop}\label{H^0(S_A)-ne-H^0(S^1_A^1)}
		Let $S$ be an inverse semigroup and $(\alpha,\lambda)$ an $S$-module structure on $A$. Then
		\begin{align}
		H^0(S^1,A^1)&=0,\label{H^0(S^1_A^1)}\\
		H^1(S^1,A^1)&\cong Z^1(S^1,A^1)=\{f:S\to A\mid\lambda_s(f(t))f(st)\m f(s)=\alpha(stt\m s\m)\}.\label{H^1(S^1_A^1)}
		\end{align}
		Moreover, if $S$ is a monoid, then
		\begin{align}
		H^0(S,A)&\cong\{a\in\cU A\mid \lambda_s(a)a\m=\alpha(ss\m),\  \forall s\in S\},\label{H^0(S_A)}\\
		H^1(S,A)&\cong Z^1(S^1,A^1)/\{f:S\to A\mid \exists a\in\cU A:\ f(s)=\lambda_s(a)a\m,\  \forall s\in S\}.\label{H^1(S_A)}
		\end{align}
	\end{prop}
	\begin{proof}
		Indeed, in the monoid case each $f\in\Hom(C_0(S),A)$ is identified with $f(\phantom{x})\in \cU A$ by \cref{V_0(S)_e}, and $ (\delta_1^0f) (s)=\lambda_s(f(\phantom{x}))f(\phantom{x})\m$ by \cref{partial_1'}, whence \cref{H^0(S_A)}. Similarly $\Hom(D_0(S^1),A^1)=\Hom(C_0(S^1),A^1)=\cU{A^1}=\{1_A\}$, giving \cref{H^0(S^1_A^1),H^1(S^1_A^1)}. In the monoid case $\Hom(C_1(S),A)\cong\Hom(D_1(S^1),A^1)$ as showed in the proof of \cref{H^n(S_A)=H^n(S^1_A^1)}, which explains \cref{H^1(S_A)}.
	\end{proof}

	\subsection{An interpretation of \texorpdfstring{$H^2(S^1,A^1)$}{H2(S1,A1)}}
	
	Observe that $f\in Z^2(S^1,A^1)$ if and only if it satisfies \cref{2-cocycle-id-for-f} of the definition of a twisted $S$-module. As to \cref{f(se_e)-and-f(e_es)}, we first notice that it can be replaced by a ``weaker-looking'' condition.
	
	\begin{lem}\label{f(e_e)-for-f(se_e)-and-f(e_es)}
		Let $f\in Z^2(S^1,A^1)$. Then each one of the equalities  
		\begin{align*}
		f(e,es)&=\alpha(ess\m)\mbox{ for all $s\in S$ and $e\in E(S)$},\\
		f(se,e)&=\alpha(ses\m)\mbox{ for all $s\in S$ and $e\in E(S)$}
		\end{align*}
		is equivalent to $f(e,e)=\alpha(e)$ for all $e\in E(S)$.
	\end{lem}
	\begin{proof}
		Clearly, $f(e,e)=\alpha(e)$ is a particular case of $f(e,es)=\alpha(ess\m)$, as well as of $f(se,e)=\alpha(ses\m)$. Conversely, assuming $f(e,e)=\alpha(e)$ and writing \cref{2-cocycle-id-for-f} for $t=u=e$, we get $\lambda_s(\alpha(e))f(s,e)=f(s,e)f(se,e)$,  whence $f(se,e)=\alpha(ses\m)$. Similarly $f(e,es)=\alpha(ess\m)$ follows from \cref{2-cocycle-id-for-f} for the triple $(e,e,s)$.
	\end{proof}
	
	\begin{lem}\label{f-cohom-to-normalized}
		For each $f\in Z^2(S^1,A^1)$ there is $g\in C^1(S^1,A^1)$, such that $\tilde f=f\cdot\delta_2^1g$ is a twisting related to the $S$-module $A$.
	\end{lem}
	\begin{proof}
		Setting $g(s)=f(s,s\m)\m$, we see that $g(s)\in A_{\alpha(ss\m)}$, so $g\in C^1(S^1,A^1)$. Moreover,
		$$
		(\delta_2^1g)(s,t)=\lambda_s(g(t))g(st)\m g(s)=\lambda_s(f(t,t\m)\m)f(st,t\m s\m)f(s,s\m)\m,
		$$
		hence 
		$$
		(\delta_2^1g)(e,e)=\lambda_e(f(e,e)\m)f(e,e)f(e,e)\m=\alpha(e)f(e,e)\m=f(e,e)\m.
		$$
		Therefore, if $\tilde f=f\cdot\delta_2^1g$, then $\tilde f(e,e)=f(e,e)f(e,e)\m=\alpha(e)$. It remains to apply \cref{f(e_e)-for-f(se_e)-and-f(e_es)}.
	\end{proof}
	
	\begin{prop}\label{H^2<->twistings}
		There is a one-to-one correspondence between the elements of $H^2(S^1,A^1)$ and the equivalence classes of twistings related to the $S$-module $A$.
	\end{prop}
	\begin{proof}
		It follows from \cref{f-cohom-to-normalized} that each class $[f]\in H^2(S^1,A^1)$ contains the twisting $\tilde f$ related to $(\alpha,\lambda)$. Now by \cref{eq-tw-mod-f'} of the definition of equivalent twisted $S$-modules two twistings are equivalent if and only if they are cohomologous as elements of $Z^2(S^1,A^1)$.
	\end{proof}
	
	\subsection{The groups \texorpdfstring{$H^n_\le(S^1,A^1)$}{Hn(S1,A1)}}\label{sec-H^n_<=}
	
	It was proved in~\cite[Lemma 3.28]{DK2} that the twisting $f$ of a Sieben's twisted $S$-module is order-preserving in the sense that $f(s,t)\le f(s',t')$ for $s\le s'$ and $t\le t'$. The converse also holds: if $f$ is order-preserving, then it satisfies the Sieben's condition \ref{Sieben's-condition}, as
	\begin{align*}
	\alpha(ess\m)&=f(e,es)\le f(e,s)\impl \alpha(ess\m)=\alpha(ess\m)f(e,s)=f(e,s),\\
	\alpha(ses\m)&=f(se,e)\le f(s,e)\impl \alpha(ses\m)=\alpha(ses\m)f(s,e)=f(s,e).
	\end{align*}
	
	We shall say that $f\in C^n(S^1,A^1)$, $n\ge 1$, {\it is order-preserving}, if
	$$
	s_1\le t_1,\dots,s_n\le t_n\impl f(s_1,\dots,s_n)\le f(t_1,\dots,t_n).
	$$
	Since $\le$ respects multiplication, the order-preserving cochains form a subgroup of $C^n(S^1,A^1)$. We denote it by $C^n_\le(S^1,A^1)$. Note also that $\delta_2^nf$ preserves the order, whenever $f$ does. Thus, we obtain the cochain complex
	$$
	C^1_\le(S^1,A^1)\overset{\delta_2^1}{\to}\dots\overset{\delta_2^{n-1}}{\to}C^n_\le(S^1,A^1)\overset{\delta_2^n}{\to}\dots
	$$
	We would like to add one more term on the left to this sequence, whose definition is motivated by the following results, in which $S$ is assumed to be an inverse semigroup and $A$ is an $S$-module.
	
	\begin{rem}\label{H^0(S_A)-as-f:E(S)->A}
		The group $H^0(S,A)$ is isomorphic to the group of functions $f:E(S)\to A$, such that $f(e)\in A_{\alpha(e)}$ and
		\begin{align}\label{lambda_s(f(e))=f(ses-inv)}
		\lambda_s(f(e))=f(ses\m)
		\end{align}
		for all $s\in S$ and $e\in E(S)$.
	\end{rem}
	\noindent The isomorphism is explained in~\cite[Proposition 5.5]{Lausch}. We would like to note that in the monoid case a function $f:E(S)\to A$ with the properties above is determined by $f(1_S)\in\cU A$, as $f(e)=f(e\cdot 1_S\cdot e\m)=\lambda_e(f(1_S))=\alpha(e)f(1_S)$. Moreover, $\lambda_s(f(1_S))=f(s\cdot 1_S\cdot s\m)=f(ss\m)=\alpha(ss\m)f(1_S)$, so $\lambda_s(f(1_S))f(1_S)\m=\alpha(ss\m)$. Conversely, if $a\in \cU A$  is such that $\lambda_s(a)a\m=\alpha(ss\m)$, then set $f(e)=\alpha(e)a$ and notice that $f(ses\m)=\alpha(ses\m)a=\alpha(se(se)\m)a=\lambda_{se}(a)=\lambda_s(\alpha(e)a)=\lambda_s(f(e))$. This relates the result of the remark with \cref{H^0(S_A)}.
	
	\begin{lem}\label{lambda_s(f(e))=f(ses-inv)<=>f-ord-pres-and-delta^0f=0}
		Let $f:E(S)\to A$ with $f(e)\in A_{\alpha(e)}$   for all $e\in E(S)$. Then $f$ satisfies \cref{lambda_s(f(e))=f(ses-inv)} if and only if $f$ is order-preserving and
		\begin{align}\label{lambda_s(f(s-inv-s))=f(ss-inv)}
		\lambda_s(f(s\m s))=f(ss\m)
		\end{align}
		for all $s\in S$.
	\end{lem}
	\begin{proof}
		Suppose \cref{lambda_s(f(e))=f(ses-inv)}. Taking $e=s\m s$, we get \cref{lambda_s(f(s-inv-s))=f(ss-inv)}. Using \cref{lambda_s(f(e))=f(ses-inv)} with $s=e'\in E(S)$, one sees that $f(ee')=\alpha(e')f(e)\le f(e)$, so $f$ is order-preserving. 
		
		Conversely, if $f$ is order-preserving, then $f(ee')\le f(e)$, so 
		\begin{align}\label{f(ee')=alpha(e')f(e)}
		f(ee')=f(ee')f(ee')\m f(e)=\alpha(ee')f(e)=\alpha(e')f(e).
		\end{align}
		Assuming additionally \cref{lambda_s(f(s-inv-s))=f(ss-inv)}, one gets 
		\begin{align*}
		f(ses\m)&=f(se(se)\m)=\lambda_{se}(f((se)\m se))=\lambda_s(\alpha(e)f(es\m s))\\
		&=\lambda_s(\alpha(s\m s)f(e))=\lambda_s(\lambda_{s\m s}(f(e)))=\lambda_s(f(e)).
		\end{align*}
	\end{proof}
	
	Define $C^0_\le(S^1,A^1)$ to be the abelian group of order-preserving functions $f:E(S)\to A$, such that $f(e)\in A_{\alpha(e)}$ for all $e\in E(S)$. Given $f\in C^0_\le(S^1,A^1)$, set 
	\begin{align}\label{delta_2^0f(s)}
	(\delta_2^0f)(s)=\lambda_s(f(s\m s))f(ss\m)\m.
	\end{align}
	Clearly, $\delta_2^0f\in C^1_\le(S^1,A^1)$.
	
	\begin{lem}\label{lem-f(e)f(e')-inv}
		For any $f\in C^0_\le(S^1,A^1)$ and $e,e'\in E(S)$ one has $f(e)f(e')\m=\alpha(ee')$. 
	\end{lem}
	\begin{proof}
		Indeed, using \cref{f(ee')=alpha(e')f(e)} and the facts that $f(e)\in A_{\alpha(e)}$, $f(e')\in A_{\alpha(e')}$, we obtain
		\begin{align*}
		f(e)f(e')\m=\alpha(e')f(e)(\alpha(e)f(e'))\m=f(ee')f(ee')\m=\alpha(ee').
		\end{align*}
	\end{proof}
	
	\begin{lem}\label{lem-delta_2^1-circ-delta_2^0}
		The composition $\delta_2^1\circ\delta_2^0$ is zero.
	\end{lem}
	\begin{proof}
		Let $f\in C^0_\le(S^1,A^1)$ and $s,t\in S$. Then by \cref{eq-delta_2^nf}
		$$
		(\delta_2^1(\delta_2^0f))(s,t)=\lambda_s((\delta_2^0f)(t))(\delta_2^0f)(st)\m (\delta_2^0f)(s).
		$$
		In view of \cref{delta_2^0f(s)} we have
		\begin{align*}
		\lambda_s((\delta_2^0f)(t))&=\lambda_s(\lambda_t(f(t\m t))f(tt\m)\m)=\lambda_{st}(f(t\m t))\lambda_s(f(tt\m)\m),\\
		(\delta_2^0f)(st)\m&=\lambda_{st}(f(t\m s\m st)\m) f(stt\m s\m),\\
		(\delta_2^0f)(s)&=\lambda_s(f(s\m s))f(ss\m)\m.
		\end{align*}
		Hence,
		\begin{align*}
		(\delta_2^1(\delta_2^0f))(s,t)&=f(stt\m s\m)f(ss\m)\m\\
		&\lambda_{st}(f(t\m t)f(t\m s\m st)\m)\\
		&\lambda_s(f(s\m s)f(tt\m)\m).
		\end{align*}
		By \cref{lem-f(e)f(e')-inv}
		\begin{align*}
		f(stt\m s\m)f(ss\m)\m&=\alpha(stt\m s\m),\\
		f(t\m t)f(t\m s\m st)\m&=\alpha(t\m s\m st),\\
		f(s\m s)f(tt\m)\m&=\alpha(s\m stt\m).
		\end{align*}
		It remains to apply \cref{lambda_s(alpha(e))} of the definition of a twisted $S$-module to get
		$$
		(\delta_2^1(\delta_2^0f))(s,t)=\alpha(stt\m s\m).
		$$
	\end{proof}
	
	As a consequence we get
	\begin{prop}\label{C_le-complex}
		The sequence
		\begin{align}\label{eq-complex-C_le}
		C^0_\le(S^1,A^1)\overset{\delta_2^0}{\to}\dots\overset{\delta_2^{n-1}}{\to}C^n_\le(S^1,A^1)\overset{\delta_2^n}{\to}\dots
		\end{align}
		is a cochain complex of abelian groups.
	\end{prop}
	
	\begin{defn}\label{defn-H_le}
		The groups of $n$-cocycles, $n$-coboundaries and $n$-cohomologies of \cref{eq-complex-C_le} will be denoted by $Z^n_\le(S^1,A^1)$, $B^n_\le(S^1,A^1)$ and $H^n_\le(S^1,A^1)$, respectively ($n\ge 0$).
	\end{defn}
	
	\begin{prop}\label{H^0_le(S^1_A^1)-cong-H^0(S_A)}
		The group $H^0_\le(S^1,A^1)$ is isomorphic to $H^0(S,A)$.
	\end{prop}
	\begin{proof}
		This is explained by \cref{lambda_s(f(e))=f(ses-inv)<=>f-ord-pres-and-delta^0f=0,delta_2^0f(s),H^0(S_A)-as-f:E(S)->A}.
	\end{proof}
	
	\begin{prop}\label{H^1_le(S^1_A^1)-cong-H^1(S^1_A^1)}
		One has $Z^1_\le(S^1,A^1)=Z^1(S^1,A^1)$, so $H^1_\le(S^1,A^1)$ is isomorphic to
		\begin{align}\label{Z^1(S^1_A^1)-modulo-B^1_le(S^1_A^1)}
		Z^1(S^1,A^1)/\{f:S\to A\mid \exists g\in C^0_\le(S^1,A^1):\  f(s)=\lambda_s(g(s\m s))g(ss\m)\m\}.
		\end{align}
		In particular, it is isomorphic to $H^1(S,A)$, when $S$ is a monoid.
	\end{prop}
	\begin{proof}
		Let $f\in Z^1(S^1,A^1)$. Applying the $1$-cocycle identity to the pair $(e,e)$, where $e\in E(S)$, we get $\alpha(e)(f(e))f(e)\m f(e)=\alpha(e)$, that is $f(e)=\alpha(e)$. Now writing the same identity for the pair $(e,s)$, we have $\alpha(e)f(s)f(es)\m f(e)=\alpha(ess\m)$, yielding $f(es)=\alpha(e)f(s)\le f(s)$, so $f$ is order-preserving. This shows that $Z^1_\le(S^1,A^1)=Z^1(S^1,A^1)$, proving \cref{Z^1(S^1_A^1)-modulo-B^1_le(S^1_A^1)}.
		
		If $S$ is a monoid and $g\in C^0_\le(S^1,A^1)$, then $g(e)=\alpha(e)g(1_S)$ by \cref{f(ee')=alpha(e')f(e)}, so $g$ is identified with $g(1_S)\in \cU A$. The result now follows from \cref{H^1(S_A)}.
	\end{proof}
	
	
	\section{Partial group cohomology with values in non-unital partial modules}\label{sec-non-unital-H^n}
	\begin{lem}\label{lem-M(S)-is-comm}
		Let $S$ be a commutative semigroup and $S^2=S$. Then the multipliers of $S$ commute with each other and with the elements of $S$.
	\end{lem}
	\begin{proof}
		It was proved in~\cite[Remark 5.2]{DK2} that $ws=sw$ for all $w\in\cM S$ and $s\in S$. Now if $w',w''\in\cM S$, then
		$$
		(w'w'')s=w'(w''s)=(w''s)w'=(sw'')w'=s(w''w')=(w''w')s.
		$$
		Similarly $s(w'w'')=s(w''w')$.
	\end{proof}
	
	Let $G$ be a group and $A$ a semilattice of groups. If $A$ is commutative, then any twisted partial action of $G$ on $A$ splits into a partial action $\0=\{\0_x:\cD_{x\m}\to\cD_x\}_{x\in G}$ of $G$ on $A$ and a {\it twisting related to $(A,\0)$}, i.\,e. a collection $w=\{w_{x,y}\}_{x,y\in G}$ of invertible multipliers of $\cD_x\cD_{xy}$, satisfying $w_{1,x}=w_{x,1}=\id_{\cD_x}$ and 
	\begin{align}\label{eq-2-coc-id-comm}
	\0_x(aw_{y,z})w\m_{xy,z}w_{x,yz}w\m_{x,y}=\0_x(a),\ \ a\in\cD_{x\m}\cD_y\cD_{yz}.
	\end{align}
	Here we used \cref{lem-M(S)-is-comm} restricting $w_{xy,z}$, $w_{x,yz}$ and $w_{x,y}$ to the (idempotent) ideal $\cD_x\cD_{xy}\cD_{xyz}$ which contains both $\0_x(aw_{y,z})$ and $\0_x(a)$. Observe that \cref{eq-2-coc-id-comm} is the same as
	\begin{align*}
	\0_x(\0_{x\m}(a)w_{y,z})w\m_{xy,z}w_{x,yz}w\m_{x,y}=a,\ \ a\in\cD_x\cD_{xy}\cD_{xyz}.
	\end{align*}
	Now $(\0,w)$ is equivalent to $(\0',w')$ if and only if $\0=\0'$ and $w$ is {\it equivalent} to $w'$ in the sense that there exists $\e=\{\e_x\in\cU{\cM{\cD_x}}\}_{x\in G}$, such that
	$$
	aw'_{x,y}=\0_x(\0_{x\m}(a)\e_y)\e\m_{xy}\e_xw_{x,y},\ \ a\in\cD_x\cD_{xy}.
	$$
	This motivates us to introduce the following notions.
	\begin{defn}\label{defn-part-G-mod}
		Let $G$ be a group. A {\it partial $G$-module} is a semilattice of abelian groups $A$ with a partial action of $G$ on $A$.
	\end{defn}
	
	Given a partial $G$-module $(A,\0)$ and $x_1,\dots,x_n\in G$, we shall write $\cD_{(x_1,\dots,x_n)}$ for $\cD_{x_1}\cD_{x_1x_2}\dots\cD_{x_1\dots x_n}$.
	
	\begin{defn}\label{defn-part-n-cochain}
		Let $(A,\0)$ be a partial $G$-module and $n\ge 1$. A {\it partial $n$-cochain} of $G$ with values in $A$ is a collection $w=\{w(x_1,\dots,x_n)\mid x_1,\dots,x_n\in G\}$, where $w(x_1,\dots,x_n)\in\cU{\cM{\cD_{(x_1,\dots,x_n)}}}$. By a {\it partial $0$-cochain} of $G$ with values in $A$ we mean $w\in\cU{\cM A}$.
	\end{defn}
	It follows from \cref{lem-M(S)-is-comm} that partial $n$-cochains form an abelian group under pointwise multiplication. We denote this group by $C^n(G,A)$.
	
	\begin{rem}\label{rem-C^n-inv-direct-prod}
		Observe that $C^n(G,A)$, $n\ge 1$, is the group of units of 
		$$
		\prod_{(x_1,\dots,x_n)\in G^n}\cM{\cD_{(x_1,\dots,x_n)}}.
		$$
	\end{rem}
	
	\begin{defn}\label{defn-delta^nw}
		Given $n\ge 1$, $w\in C^n(G,A)$ and $a\in\cD_{(x_1,\dots,x_{n+1})}$, define
		\begin{align}
		(\delta^nw)(x_1,\dots,x_{n+1})a&=\0_{x_1}(\0_{x\m_1}(a)w(x_2,\dots,x_{n+1}))\notag\\
		&\prod_{i=1}^n w(x_1,\dots,x_ix_{i+1},\dots,x_{n+1})^{(-1)^i}\notag\\
		&w(x_1,\dots,x_n)^{(-1)^{n+1}}.\label{eq-delta^nw}
		\end{align}
		For $w\in C^0(G,A)$ and $a\in\cD_x$ set 
		\begin{align}\label{eq-delta^0w}
		(\delta^0w)(x)a&=\0_x(\0_{x\m}(a)w)w\m.
		\end{align}
	\end{defn}
	Observe that $\0_{x\m_1}(a)\in\cD_{x\m_1}\cD_{(x_2,\dots,x_{n+1})}$, so $w(x_2,\dots,x_{n+1})$ is applicable in \cref{eq-delta^nw}. The result $\0_{x\m_1}(a)w(x_2,\dots,x_{n+1})$ belongs to $\cD_{x\m_1}\cD_{(x_2,\dots,x_{n+1})}$, since the latter is an idempotent ideal. Therefore, $\0_{x_1}(\0_{x\m_1}(a)w(x_2,\dots,x_{n+1}))$ is an element of $\cD_{(x_1,\dots,x_{n+1})}$. So, the rest of the multipliers in \cref{eq-delta^nw} are obviously applicable and $(\delta^nw)(x_1,\dots,x_{n+1})a\in\cD_{(x_1,\dots,x_{n+1})}$.
	
	\begin{lem}\label{lem-delta^nw-multiplier}
		For all $n\ge 0$ and $x_1,\dots,x_{n+1}\in G$ the map $(\delta^nw)(x_1,\dots,x_{n+1})$ is a multiplier of $\cD_{(x_1,\dots,x_{n+1})}$,  whose right action coincides with the left one.
	\end{lem}
	\begin{proof}
		Let $n\ge 1$. According to \cref{eq-delta^nw} we may write 
		\begin{align}\label{eq-delta^nw-using-w'w''}
		(\delta^nw)(x_1,\dots,x_{n+1})a=\0_{x_1}(\0_{x\m_1}(a)w')w'',
		\end{align}
		where
		\begin{align*}
		w'&=w(x_2,\dots,x_{n+1}),\\
		w''&=\left(\prod_{i=1}^n w(x_1,\dots,x_ix_{i+1},\dots,x_{n+1})^{(-1)^i}\right)w(x_1,\dots,x_n)^{(-1)^{n+1}}.
		\end{align*}
		We shall check the equality
		$$
		(\delta^nw)(x_1,\dots,x_{n+1})(ab)=((\delta^nw)(x_1,\dots,x_{n+1})a)b
		$$
		for $a,b\in\cD_{(x_1,\dots,x_{n+1})}$. The other two properties of a multiplier are proved similarly. Taking into account \cref{lem-M(S)-is-comm,eq-delta^nw-using-w'w''}, we have
		\begin{align*}
		(\delta^nw)(x_1,\dots,x_{n+1})(ab)&=\0_{x_1}(\0_{x\m_1}(ab)w')w''=\0_{x_1}(\0_{x\m_1}(a)\0_{x\m_1}(b)w')w''\\
		&=\0_{x_1}(\0_{x\m_1}(a)w'\0_{x\m_1}(b))w''=\0_{x_1}(\0_{x\m_1}(a)w')bw''\\
		&=\0_{x_1}(\0_{x\m_1}(a)w')w''b=((\delta^nw)(x_1,\dots,x_{n+1})a)b.
		\end{align*}
		
		The case $n=0$ uses the same idea (take $w'=w$ and $w''=w\m$).
	\end{proof}
	
	\begin{lem}\label{lem-w-inside-0}
		For all $a\in\cD_{x\m}\cD_{(y_1,\dots,y_n)}$, $w\in\cM{\cD_{x\m}\cD_{(y_1,\dots,y_n)}}$ and $w'\in\cM{\cD_{(x,y_1,\dots,y_n)}}$  one has $\0_{x\m}(\0_x(aw)w')=\0_{x\m}(\0_x(a)w')w$.
	\end{lem}
	\begin{proof} By \cite[Proposition 2.7]{DE} the pair of maps  $a \mapsto \0_{x\m}(\0_x(a)w')$ and 
		$a \mapsto \0_{x\m}(w'\0_x(a))$ defines a multiplier $\bar{w}'$ of the idempotent ideal $\cD_{x\m}\cD_{(y_1,\dots,y_n)}$. Then our statement  transforms into the equality 
		$(aw)\bar{w}' = (a\bar{w}')w$, which holds thanks to \cref{lem-M(S)-is-comm}.
	\end{proof}
	
	\begin{lem}\label{lem-delta^n-homo}
		For all $n\ge 0$ the map $\delta^n$ is a homomorphism $C^n(G,A)\to C^{n+1}(G,A)$.
	\end{lem}
	\begin{proof}
		We shall prove that $\delta^n$ is a homomorphisms of monoids
		$$
		C^n(G,A)\to\prod_{(x_1,\dots,x_{n+1})\in G^{n+1}}\cM{\cD_{(x_1,\dots,x_{n+1})}}.
		$$
		In view of \cref{rem-C^n-inv-direct-prod} this will imply that $\delta^n(C^n(G,A))\subseteq C^{n+1}(G,A)$.
		
		The fact that $\delta^n$ maps identity to identity is clear from \cref{eq-delta^nw}. Fix $n\ge 1$, $u,v\in C^n(G,A)$, $x_1,\dots,x_{n+1}\in G$ and $a\in\cD_{(x_1,\dots,x_{n+1})}$. We need to show that
		\begin{align}\label{eq-delta^nuv}
		(\delta^n(uv))(x_1,\dots,x_{n+1})a=(\delta^nu)(x_1,\dots,x_{n+1})(\delta^nv)(x_1,\dots,x_{n+1})a.
		\end{align}
		Using \cref{eq-delta^nw-using-w'w''}, we represent the right-hand side of \cref{eq-delta^nuv} as
		\begin{align*}
		(\delta^nu)(x_1,\dots,x_{n+1})\0_{x_1}(\0_{x\m_1}(a)v')v''=\0_{x_1}(\0_{x\m_1}(\0_{x_1}(\0_{x\m_1}(a)v')v'')u')u'',
		\end{align*}
		where $u',v'\in\cM{\cD_{(x_2,\dots,x_{n+1})}}$ and $u'',v''\in\cM{\cD_{(x_1,\dots,x_{n+1})}}$. By \cref{lem-w-inside-0}
		$$
		\0_{x\m_1}(\0_{x_1}(\0_{x\m_1}(a)v')v'')=\0_{x\m_1}(\0_{x_1}(\0_{x\m_1}(a))v'')v'=\0_{x\m_1}(av'')v'.
		$$
		Therefore,
		$$
		\0_{x_1}(\0_{x\m_1}(\0_{x_1}(\0_{x\m_1}(a)v')v'')u')u''=\0_{x_1}(\0_{x\m_1}(av'')v'u')u'',
		$$
		the latter being $\0_{x_1}(\0_{x\m_1}(a)v'u')v''u''$ in view of \cref{lem-w-inside-0}. The result now follows from the observation that $(uv)'=u'v'$ and $(uv)''=u''v''$, where $(uv)'$ and $(uv)''$ denote the ``parts'' of $(\delta^n(uv))(x_1,\dots,x_{n+1})a$ from the representation similar to \cref{eq-delta^nw-using-w'w''}.
		
		The case $n=0$ is proved analogously.
	\end{proof}
	
	Given a partial $G$-module $(A,\0)$, as it was mentioned above, there exist an $E$-unitary semigroup $S$ and an epimorphism $\kappa:S\to G$ whose kernel coincides with $\sigma$, such that \cref{lambda-from-Theta} defines an $S$-module structure $(\alpha,\lambda)$ on $A$. As to formula \cref{f-from-Theta}, it can be generalized to arbitrary $n\ge 0$. 
	
	\begin{lem}\label{lem-from-w-to-f}
		For all $n\ge 0$ there is a homomorphism from $C^n(G,A)$ to $C^n_\le(S^1,A^1)$ which maps $w$ to $f$ defined by 
		\begin{align}
		f(e)&=\alpha(e)w, & n=0,\label{eq-f(e)}\\
		f(s_1,\dots,s_n)&=\alpha(s_1\dots s_ns\m_n\dots s\m_1)w(\kappa(s_1),\dots,\kappa(s_n)), & n\ge 1,\label{eq-f(s_1...s_n)}
		\end{align}
		where $e\in E(S)$, $s_1,\dots,s_n\in S$.
	\end{lem}
	\begin{proof}
		The case $n=0$ is immediate: $\alpha$, being an isomorphism $E(S)\to E(A)$, preserves the order and hence $f$ does. Moreover, since $w\in\cU{\cM A}$, we see from \cref{eq-ws-inv,eq-f(e)} that $f(e)f(e)\m=\alpha(e)ww\m=\alpha(e)$. Thus, $f\in C^0_\le(S^1,A^1)$.
		
		When $n\ge 1$, we first note that the right-hand side of \cref{eq-f(s_1...s_n)} makes sense, as for $s_i=e_i\delta_{x_i}$, $e_i\in E(\cD_{x_i})$, one has $x_i=\kappa(s_i)$, $1\le i\le n$, and $s:=s_1\dots s_n=e\delta_x$, where $e\in E(\cD_{(x_1,\dots,x_n)})$ and $x=x_1\dots x_n$. So, $\alpha(ss\m)=e$ and thus $w(\kappa(s_1),\dots,\kappa(s_n))$ is applicable in \cref{eq-f(s_1...s_n)}. 
		
		As above for $n=0$, one easily gets from \cref{eq-ws-inv} that
		$$
		f(s_1,\dots,s_n)f(s_1,\dots,s_n)\m=\alpha(s_1\dots s_ns\m_n\dots s\m_1).
		$$
		Moreover, if $s_i\le t_i$, $1\le i\le n$, then
		$$
		f(s_1,\dots,s_n)\le f(t_1,\dots,t_n),
		$$
		because
		$$
		s_1\dots s_ns\m_n\dots s\m_1\le t_1\dots t_nt\m_n\dots t\m_1,
		$$
		$\alpha$ preserves the order and $\kappa(s_i)=\kappa(t_i)$, as $(s_i,t_i)\in\sigma$, $1\le i\le n$. This shows that $f\in C^n_\le(S^1,A^1)$.
		
		The fact that the map $w\mapsto f$ is a homomorphism is explained by the observation that in a commutative semigroup $S$ one has $e(ww')=(ew)w'=(e(ew))w'=((ew)e)w'=(ew)(ew')$ for all $e\in E(S)$ and $w,w'\in\cM S$.
	\end{proof}
	
	Recall now that we may take $S=E(A)*_\0 G$ with $\alpha:E(S)\to E(A)$ and $\kappa: S\to G$ given by \cref{kappa(e-delta_x),alpha(e-delta_1)}. Notice that for any $x\in G$ and $a\in\cD_x$ we have that $s=aa\m\delta_x$ is the unique element of $S$ such that $\kappa(s)=x$ and $\alpha(ss\m)=aa\m$. 
	
	\begin{lem}\label{lem-from-f-to-w}
		Let $n\ge 0$ and $f\in C^n_\le(S^1,A^1)$. If $n=0$, then for any $a\in A$ define
		\begin{align}\label{eq-wa}
		wa=aw=f(\alpha\m(aa\m))a.
		\end{align}
		When $n\ge 1$, given $x_1,\dots,x_n\in G$ and $a\in\cD_{(x_1,\dots,x_n)}$, choose a unique $n$-tuple $(s_1,\dots,s_n)\in S^n$, such that $\kappa(s_i)=x_1\dots x_i$ and $\alpha(s_is\m_i)=aa\m$, $1\le i\le n$. Then set\footnote{When $n=1$, the right-hand term  of \cref{eq-w(x_1...x_n)a}  is $f(s_1)a$.}
		\begin{align}\label{eq-w(x_1...x_n)a}
		w(x_1,\dots,x_n)a=aw(x_1,\dots,x_n)=f(s_1,s\m_1 s_2,\dots,s\m_{n-1} s_n)a.
		\end{align}
		This defines a homomorphism from $C^n_\le(S^1,A^1)$ to $C^n(G,A)$, $n\ge 0$.
	\end{lem}
	\begin{proof}
		For $n=0$ we note from \cref{eq-wa} using the order-preserving property of $f$ that
		\begin{align*}
		w(ab)&=f(\alpha\m(abb\m a\m))ab=f(\alpha\m(aa\m bb\m))ab\\
		&=f(\alpha\m(aa\m)\alpha\m(bb\m))ab\le f(\alpha\m(aa\m))ab=(wa)b.
		\end{align*}
		Since both $w(ab)$ and $(wa)b$ belong to the same group component $A_{aa\m bb\m}$ of $A$, they are equal. Due to the equality $wa=aw$ and the commutativity of $A$ this explains that $w$ is a multiplier of $A$. Clearly, $w$ is invertible with $w\m a=aw\m=f(\alpha\m(aa\m))\m a$. So, $w\in C^0(G,A)$. 
		
		Suppose that $f\mapsto w$ and $f'\mapsto w'$ for some $f,f'\in C^0_\le(S^1,A^1)$. As $wa\in A_{aa\m}$, 
		$$
		w'(wa)=f'(\alpha\m(aa\m))wa=f'(\alpha\m(aa\m))f(\alpha\m(aa\m))a,
		$$
		showing that $f'f\mapsto w'w$.
		
		Let $n\ge 1$. It is immediately seen that the right-hand term of \cref{eq-w(x_1...x_n)a} belongs to the ideal $\cD_{(x_1,\dots,x_n)}$. Observe that 
		\begin{align}\label{eq-f-in-A_aa-inv}
		f(s_1,s\m_1 s_2,\dots,s\m_{n-1} s_n)\in A_{\alpha(s_1s\m_1\dots s_ns\m_n)}=A_{aa\m}. 
		\end{align}
		Hence, the function $w(x_1,\dots,x_n)$ from $\cD_{(x_1,\dots,x_n)}$ to itself is a bijection, whose inverse is
		$$
		w(x_1,\dots,x_n)\m a=aw(x_1,\dots,x_n)\m=f(s_1,s\m_1 s_2,\dots,s\m_{n-1} s_n)\m a.
		$$
		
		To prove that $w(x_1,\dots,x_n)$ is a multiplier of $\cD_{(x_1,\dots,x_n)}$, it suffices to verify
		\begin{align}\label{eq-w(x_1...x_n)-multiplier}
		w(x_1,\dots,x_n)(ab)=(w(x_1,\dots,x_n)a)b
		\end{align}
		for $a,b\in\cD_{(x_1,\dots,x_n)}$. Let $(s_1,\dots,s_n)$ and $(t_1,\dots,t_n)$ be the $n$-tuples  in $S^n$ from the definition of $w$, corresponding to $a$ and $b$, respectively. Then the right-hand side of \cref{eq-w(x_1...x_n)-multiplier} equals
		$$
		f(s_1,s\m_1 s_2,\dots,s\m_{n-1} s_n)ab
		$$
		by \cref{eq-w(x_1...x_n)a}. As to the left-hand side of \cref{eq-w(x_1...x_n)-multiplier}, note that
		$$
		ab(ab)\m=aa\m bb\m=\alpha(s_is\m_i t_1t\m_1)=\alpha(t_1t\m_1 s_i(t_1t\m_1 s_i)\m)
		$$
		with $\kappa(t_1t\m_1 s_i)=\kappa(s_i)=x_1\dots,x_i$, $1\le i\le n$. Therefore, $w(x_1,\dots,x_n)(ab)$ is
		\begin{align}
		&f(t_1t\m_1 s_1,s\m_1 t_1t\m_1\cdot  t_1t\m_1  s_2,\dots,s\m_{n-1}  t_1t\m_1 \cdot t_1t\m_1 s_n)ab\notag\\
		&=f(t_1t\m_1 s_1,s\m_1 t_1t\m_1 s_2,\dots,s\m_{n-1} t_1t\m_1 s_n)ab.\label{eq-long-f-ab}
		\end{align}
		Since $t_1t\m_1 s_1\le s_1$ and $s\m_i t_1t\m_1 s_{i+1}\le s\m_i s_{i+1}$, $1\le i\le n-1$, we have
		\begin{align}\label{eq-f-le-f}
		f(t_1t\m_1 s_1,s\m_1 t_1t\m_1 s_2,\dots,s\m_{n-1} t_1t\m_1 s_n)\le f(s_1,s\m_1 s_2,\dots,s\m_{n-1} s_n). 
		\end{align}
		Taking into account \cref{eq-f-in-A_aa-inv} and the fact that the left-hand side of \cref{eq-f-le-f} belongs to
		$$
		A_{\alpha(t_1t\m_1 s_1s\m_1\dots s_ns\m_n)}=A_{aa\m bb\m},
		$$
		one sees that \cref{eq-long-f-ab} is 
		$$
		bb\m f(s_1,s\m_1 s_2,\dots,s\m_{n-1} s_n)ab=f(s_1,s\m_1 s_2,\dots,s\m_{n-1} s_n)ab.
		$$
		This concludes the proof that $w\in C^n(G,A)$.
		
		Take additionally $f'\in C^n_\le(S^1,A^1)$ and let $f'\mapsto w'$. It follows from \cref{eq-f-in-A_aa-inv} that $w(x_1,\dots,x_n)a\in A_{aa\m}$, so to apply $w'(x_1,\dots,x_n)$ to this element, one uses the same $n$-tuple $(s_1,\dots,s_n)$: 
		\begin{align*}
		w'w(x_1,\dots,x_n)a&=f'(s_1,s\m_1 s_2,\dots,s\m_{n-1} s_n)w(x_1,\dots,x_n)a\\
		&=f'f(s_1,s\m_1 s_2,\dots,s\m_{n-1} s_n)a.
		\end{align*}
		Thus, $f'f$ is mapped to $w'w$.
	\end{proof}
	
	\begin{lem}\label{C^n-isom-C^n_<=}
		The group $C^n(G,A)$ is isomorphic to $C^n_\le(S^1,A^1)$ for all $n\ge 0$.
	\end{lem}
	\begin{proof}
		We shall show that the homomorphisms from \cref{lem-from-w-to-f,lem-from-f-to-w} are inverse to each other.
		
		Let $w\in C^n(G,A)$ and $w\mapsto f\mapsto w'$. If $n=0$, then by \cref{eq-f(e),eq-wa}
		$$
		w'a=f(\alpha\m(aa\m))a=(\alpha(\alpha\m(aa\m))w)a=((aa\m) w)a=w(aa\m) a=wa, 
		$$
		so $w'=w$ by \cref{lem-M(S)-is-comm}. Consider now $n\ge 1$ and take $x_1,\dots,x_n\in G$ and $a\in\cD_{(x_1,\dots,x_n)}$. Find a unique $(s_1,\dots,s_n)\in S^n$, such that $\kappa(s_i)=x_1\dots x_i$ and $\alpha(s_is\m_i)=aa\m$, $1\le i\le n$. According to \cref{eq-w(x_1...x_n)a,eq-f(s_1...s_n)}:
		\begin{align*}
		w'(x_1,...,x_n)a&=f(s_1,s\m_1 s_2,\dots,s\m_{n-1} s_n)a\\
		&=\alpha(s_1s\m_1\dots s_ns\m_n)w(\kappa(s_1),\kappa(s\m_1 s_2),\dots,\kappa(s\m_{n-1} s_n))a\\
		&=aa\m w(x_1,x\m_1  x_1x_2,\dots,x\m_{n-1}\dots x\m_1 x_1\dots x_n)a\\
		&=w(x_1,...,x_n)a.
		\end{align*}
		
		Now take $f\in C^n_\le(S^1,A^1)$ and assume that $f\mapsto w\mapsto f'$. For $n=0$ one has
		$$
		f'(e)=\alpha(e)w=f(\alpha\m(\alpha(e)))\alpha(e)=f(e)\alpha(e)=f(e),
		$$
		as $f(e)\in A_{\alpha(e)}$. Let $n\ge 1$ and $s_1,\dots,s_n\in S$.  Set
		$$
		e_i=s_i\dots s_ns\m_n\dots s\m_i\in E(S),\ \ t_i=s_1\dots s_i\cdot e_{i+1},\ \ 1\le i\le n-1, 
		$$
		and $t_n=s_1\dots s_n$. Then $\kappa(t_i)=\kappa(s_1)\dots\kappa(s_i)$ and $\alpha(t_it\m_i)=\alpha(e_1)$, $1\le i\le n$. So, in view of \cref{eq-f(s_1...s_n),eq-w(x_1...x_n)a}:
		\begin{align*}
		f'(s_1,\dots,s_n)&=\alpha(s_1\dots s_ns\m_n\dots s\m_1)w(\kappa(s_1),\dots,\kappa(s_n))\\
		&=\alpha(e_1)f(t_1,t\m_1 t_2,\dots,t\m_{n-1} t_n).
		\end{align*}
		
		Clearly, $t_1=s_1e_2\le s_1$. Moreover,
		$$
		t\m_i t_{i+1}=e_{i+1}\cdot s\m_i\dots s\m_1 s_1\dots s_i\cdot s_{i+1}\cdot e_{i+2}\le s_{i+1},\ \ 1\le i\le n-1.
		$$
		Therefore,
		$$
		f(t_1,t\m_1 t_2,\dots,t\m_{n-1} t_n)\le f(s_1,s_2,\dots, s_n)
		$$
		and hence
		$$
		f'(s_1,\dots,s_n)\le\alpha(e_1)f(s_1,\dots,s_n)=f(s_1,\dots,s_n).
		$$
		Since both $f'(s_1,\dots,s_n)$ and $f(s_1,\dots,s_n)$ belong to the same group component $A_{\alpha(e_1)}$ of $A$, we conclude that they are equal.
	\end{proof}
	
	\begin{rem}\label{w-normalized-iff-f-normalized}
		Observe that $w\in C^2(G,A)$ satisfies $w(1,x)=w(x,1)=\id_{\cD_x}$ for all $x$ if and only if the corresponding $f\in C^2_\le(S^1,A^1)$ satisfies the Sieben's condition \ref{Sieben's-condition}.
	\end{rem}
	\noindent Indeed, if $w(1,x)=\id_{\cD_x}$ for all $x$, then $f(e,s)=\alpha(ess\m)w(1,\kappa(s))=\alpha(ess\m)$ by \cref{eq-f(s_1...s_n)}, and analogously $w(x,1)=\id_{\cD_x}$ for all $x$ implies $f(s,e)=\alpha(ses\m)$. Conversely, assume \ref{Sieben's-condition} and take $x\in G$, $a\in\cD_x$. There is $s\in S$, such that $\kappa(s)=x$ and $\alpha(ss\m)=aa\m$. Since $\kappa(ss\m)=1$, it follows that $w(1,x)a=aw(1,x)=f(ss\m,(ss\m)\m s)a=f(ss\m,s)a=\alpha(ss\m)a=a$ by \cref{eq-w(x_1...x_n)a}. Similarly $w(x,1)a=aw(x,1)=a$.
	
	\begin{lem}\label{w->f-respects-delta}
		The homomorphism from \cref{lem-from-w-to-f} respects $\delta^n$ and $\delta_2^n$, $n\ge 0$, in the sense that the following diagram
		\begin{align}\label{diag-C^n_<=(S^1A^1)-C^n(GA)}
		\begin{tikzpicture}[node distance=4cm, auto]
		\node (C_0) {$C^0(G,A)$};
		\node (C_1) [left of=C_0] {$C^1(G,A)$};
		\node (dots_C_1) [xshift=1.1cm, left of=C_1] {$\dots$};
		\node (D_0) [yshift=2.5cm, below of=C_0]{$C^0_\le(S^1,A^1)$};
		\node (D_1) [left of=D_0] {$C^1_\le(S^1,A^1)$};
		\node (dots_D_1) [xshift=1.1cm, left of=D_1] {$\dots$};
		\draw[->] (C_0) to node[above] {$\delta^0$} (C_1);
		\draw[->] (C_1) to node[above] {$\delta^1$} (dots_C_1);
		\draw[->] (D_0) to node[above] {$\delta_2^0$} (D_1);
		\draw[->] (D_1) to node[above] {$\delta_2^1$} (dots_D_1);
		\draw[-] (C_0) to node {$\wr$} (D_0);
		\draw[-] (C_1) to node {$\wr$} (D_1);
		\end{tikzpicture}
		\end{align}
		commutes.
	\end{lem}
	\begin{proof}
		Let $w\in C^n(G,A)$. Suppose that $w\mapsto f\in C^n_\le(S^1,A^1)$ and $\delta^nw\mapsto f'\in C^{n+1}_\le(S^1,A^1)$. We need to prove that $f'=\delta_2^nf$. 
		
		Consider the case $n=0$. By \cref{delta_2^0f(s),eq-f(e),lambda-from-Theta,eq-delta^0w,eq-f(s_1...s_n)}
		\begin{align*}
		(\delta_2^0f)(s)&=\lambda_s(f(s\m s))f(ss\m)\m=\0_{\kappa(s)}(\alpha(s\m s)f(s\m s))f(ss\m)\m\\
		&=\0_{\kappa(s)}(\alpha(s\m s)w)\alpha(ss\m)w\m=\0_{\kappa(s)}(\alpha(s\m s)w)w\m\\
		&=\0_{\kappa(s)}(\0_{\kappa(s)\m}(\alpha(ss\m))w)w\m=\alpha(ss\m)(\delta^0w)(\kappa(s))=f'(s).
		\end{align*}
		Here we also used the fact that $\0_{\kappa(s)}(\alpha(s\m s))=\lambda_s(\alpha(s\m s))=\alpha(ss\m)$.

		Let $n\ge 1$. Given $t_1,\dots,t_k\in S$, we shall denote for brevity
		$$
		e(t_1,\dots,t_k)=\alpha(t_1\dots t_kt_k\m\dots t\m_1)\in E(A).
		$$
		
		Using \cref{eq-f(s_1...s_n)}, we see that the factor $f(s_1,\dots,s_is_{i+1},\dots,s_{n+1})^{(-1)^i}$ in \cref{eq-delta_2^nf} equals
		\begin{align}\label{eq-f(s_1...s_is_(i+1)...s_(n+1))}
		e(s_1,\dots,s_{n+1})w(\kappa(s_1),\dots,\kappa(s_is_{i+1}),\dots,\kappa(s_{n+1}))^{(-1)^i},\ \ 1\le i\le n.
		\end{align}
		Since \cref{eq-delta_2^nf} contains $\lambda_{s_1}(f(s_2,\dots,s_{n+1}))\in A_{e(s_1,\dots,s_{n+1})}$, we may easily remove $e(s_1,\dots,s_{n+1})$ from \cref{eq-f(s_1...s_is_(i+1)...s_(n+1))}. Moreover, we may remove $e(s_1,\dots,s_n)$ from
		$$
		f(s_1,\dots,s_n)^{(-1)^{n+1}}=e(s_1,\dots,s_n)w(\kappa(s_1),\dots,\kappa(s_n))^{(-1)^{n+1}},
		$$
		as $e(s_1,\dots,s_{n+1})\le e(s_1,\dots,s_n)$. Taking into account \cref{lambda-from-Theta}, we come to
		\begin{align}
		(\delta_2^nf)(s_1,\dots,s_{n+1})&=\0_{\kappa(s_1)}(\alpha(s\m_1 s_1)e(s_2,\dots,s_{n+1})w(\kappa(s_2),\dots,\kappa(s_{n+1})))\notag\\
		&\prod_{i=1}^nw(\kappa(s_1),\dots,\kappa(s_is_{i+1}),\dots,\kappa(s_{n+1}))^{(-1)^i}\notag\\
		&w(\kappa(s_1),\dots,\kappa(s_n))^{(-1)^{n+1}}.\label{eq-delta_2^nf-from-w}
		\end{align}
		Now observe using \cref{lambda-from-Theta} that
		\begin{align*}
		\alpha(s\m_1 s_1)e(s_2,\dots,s_{n+1})&=e(s\m_1,s_1,\dots,s_{n+1})\\
		&=\lambda_{s\m_1}(e(s_1,\dots,s_{n+1}))\\
		&=\lambda_{s\m_1}(\alpha(s_1s\m_1) e(s_1,\dots,s_{n+1}))\\
		&=\0_{\kappa(s_1)\m}(e(s_1,\dots,s_{n+1})),
		\end{align*}
		the latter showing in view of \cref{eq-delta^nw} that the right-hand side of \cref{eq-delta_2^nf-from-w} is exactly
		$$
		e(s_1,\dots,s_{n+1})(\delta^nw)(\kappa(s_1),\dots,\kappa(s_{n+1}))=f'(s_1,\dots,s_{n+1}).
		$$
	\end{proof}
	
	\begin{cor}\label{cor-C^n(G,A)-cochain-compl}
		The sequence
		\begin{align}\label{C^0(GA)->...}
		C^0(G,A)\overset{\delta^0}{\to}\dots \overset{\delta^{n-1}}{\to}C^n(G,A)\overset{\delta^n}{\to}\dots
		\end{align}
		is a cochain complex of abelian groups. 
	\end{cor}
    \begin{proof}
    	This is explained by \cref{w->f-respects-delta,C_le-complex}.
    \end{proof}
	
	\begin{defn}\label{H^n(GA)-defn}
		The complex \cref{C^0(GA)->...} naturally defines the groups $Z^n(G,A)=\ker{\delta^n}$, $B^n(G,A)=\im{\delta^{n-1}}$ and $H^n(G,A)=Z^n(G,A)/B^n(G,A)$ of {\it partial $n$-cocycles}, {\it $n$-coboundaries} and {\it $n$-cohomologies} of $G$ with values in $A$, $n\ge 1$ ($H^0(G,A)=Z^0(G,A)=\ker{\delta^0}$).
	\end{defn}
	
	The next fact immediately follows from \cref{C^n-isom-C^n_<=,w->f-respects-delta}.
	\begin{thrm}\label{H^n(GA)-cong-H^n_<=(SA)}
		There are isomorphisms of groups $Z^n(G,A)\cong Z^n_\le(S^1,A^1)$ and $B^n(G,A)\cong B^n_\le(S^1,A^1)$. In particular, $H^n(G,A)$ is isomorphic to $H^n_\le(S^1,A^1)$ for all $n\ge 0$.
	\end{thrm}
	
	\begin{cor}\label{w-cohom-to-twisting}
		There is a one-to-one correspondence between the elements of $H^2(G,A)$ and the equivalence classes of twistings related to $(A,\0)$.
	\end{cor}
    \begin{proof}
    	By \cref{H^n(GA)-cong-H^n_<=(SA)} each class $[w]\in H^2(G,A)$ corresponds to $[f]\in H^2_\le(S^1,A^1)$, and by \cref{f-cohom-to-normalized} there is $g\in C^1(S^1,A^1)$, such that $\tilde f=f\cdot \delta_2^1g$ is a twisting related to $(\alpha,\lambda)$. It is seen by the proof of \cref{f-cohom-to-normalized} that $g(s)$ can be chosen to be $f(s,s\m)\m$, so $g$ preserves the order, as $f$ does. Therefore, $\tilde f\in Z^2_\le(S^1,A^1)$ and $[f]=[\tilde f]$ in $H^2_\le(S^1,A^1)$. Since $\tilde f$ is order-preserving, it satisfies Sieben's condition \ref{Sieben's-condition} as it was observed at the beginning of \cref{sec-H^n_<=}. By \cref{w-normalized-iff-f-normalized} there is a twisting $\tilde w$ related to $(A,\0)$, such that $[\tilde w]$ corresponds to $[\tilde f]$. Thus, $[w]=[\tilde w]$.
    \end{proof}
	
	\begin{cor}\label{H^n_le-cong-H^n-for-F-inverse}
		Let $S$ be a max-generated $F$-inverse monoid (see~\cite[p. 196]{Lawson2002}) and $A$ an $S$-module. Then $H^n_\le(S^1,A^1)\cong H^n(S,A)$.
	\end{cor}
    \begin{proof}
    	By~\cite[Proposition 9.1]{DK2}, up to an isomorphism,  each $S$-module $(\alpha,\lambda)$ on $A$ comes from the partial $\cG S$-module $(A,\0)$ defined by \cref{D-from-Lambda,0-from-Lambda},   and we may assume that $S=E(A)*_\0\cG S$. Hence, $H^n_\le(S^1,A^1)\cong H^n(\cG S,A)$ by \cref{H^n(GA)-cong-H^n_<=(SA)}.  Observe from \cref{D-from-Lambda} that each $\cD_x$ is a monoid with identity $1_x=\alpha(\max x(\max x)\m)$, where $\max x$ is the maximum element of the class $x\in\cG S$. The idempotents $1_x$ generate $E(A)$, as the elements $\max x$ generate $S$.  Therefore, $(A,\0)$ is an inverse partial $\cG S$-module in the sense of~\cite[Definition 3.15]{DK}. By~\cite[Corollaries 3.34 and 4.12]{DK} one has $H^n(S,A)\cong H^n(\cG S,A)$ (see also~\cite[Definition 4.1]{DK}).\footnote{Notice that in \cite{DK} the H. Lausch's cohomology group $H^n(S,A)$ is denoted by $H^n_S(A)$.}
    \end{proof}
	
	\section{Extensions of semilattices of abelian groups and \texorpdfstring{$H^2(G,A)$}{H\texttwosuperior(G,A)}}\label{sec-ext-abel}
	
	It was proved in~\cite[Theorem 6.12]{DK2} that any admissible extension $A\overset{i}{\to}U\overset{j}{\to}G$ induces a twisted partial action $\Theta$ of $G$ by $A$, as soon as one fixes a refinement $A\overset{i}{\to}U\overset{\pi}{\to}S\overset{\kappa}{\to}G$ of $U$ together with an order-preserving transversal $\rho$ of $\pi$. We shall show that the change of $U$ by an equivalent one, as well as a change of $\rho$, leads to an equivalent $\Theta$.
	
	\begin{lem}\label{equiv-ext-to-equiv-tw-pact}
		Suppose that $A\overset{i}{\to}U\overset{j}{\to}G$ and $A\overset{i'}{\to}U'\overset{j'}{\to}G$ is a pair of equivalent admissible extensions of $A$ by $G$. Then any two refinements $A\overset{i}{\to}U\overset{\pi}{\to}S\overset{\kappa}{\to}G$ and $A\overset{i'}{\to}U'\overset{\pi'}{\to}S'\overset{\kappa'}{\to}G$ of $U$ and $U'$ with order-preserving transversals $\rho$ and $\rho'$ of $\pi$ and $\pi'$ induce equivalent twisted partial actions $\Theta$ and $\Theta'$ of $G$ on $A$.
	\end{lem}
	\begin{proof}
		Let $\mu:U\to U'$ be an isomorphism defining the equivalence. There is a homomorphism $\nu:S\to S'$ making the diagram \cref{refined-diag-ext-A-G} commute. Since $\mu$ is an isomorphism, $\nu$ is also an isomorphism. 
		
		Denote by $\Lambda=(\alpha,\lambda,f)$ and $\Lambda'=(\alpha',\lambda',f')$ the twisted module structures on $A$ over $S$ and $S'$, which come from $A\overset{i}{\to}U\overset{\pi}{\to}S$, $\rho$ and $A\overset{i'}{\to}U'\overset{\pi'}{\to}S'$, $\rho'$, respectively. Note that $\rho''=\mu\circ\rho\circ\nu\m$ is another order-preserving transversal of $\pi'$ with $\rho''\circ\nu=\mu\circ\rho$. By~\cite[Corollary 3.16]{DK2} the induced twisted $S'$-module structure $\Lambda''=(\alpha'',\lambda'',f'')$ on $A$ satisfies $\Lambda''\circ\nu=\Lambda$ in the sense that $\lambda''\circ\nu=\lambda$, $\alpha''\circ\nu|_{E(S)}=\alpha$ and $f=f''\circ(\nu\times\nu)$. 
		
		Let $\Theta''=(\0'',w'')$ be the twisted partial action of $G$ on $A$ coming from $\Lambda''$. Observe from \cref{D-from-Lambda} that
		$$
		\cD_x=\bigsqcup_{\kappa(s)=x}A_{\alpha(ss\m)}=\bigsqcup_{\kappa'\circ\nu(s)=x}A_{\alpha''\circ\nu(ss\m)}=\bigsqcup_{\kappa'\circ\nu(s)=x}A_{\alpha''(\nu(s)\nu(s)\m)}=\cD''_x.
		$$
		Moreover, $\0_x(a)=\lambda_s(a)=\lambda''_{\nu(s)}(a)$ for $s\in S$ with $x=\kappa(s)=\kappa'\circ\nu(s)$, $aa\m=\alpha(s\m s)=\alpha''(\nu(s)\m\nu(s))$, so $\0_x(a)=\0''_x(a)$ (see \cref{0-from-Lambda}). Similarly using \cref{w-from-Lambda} one proves  that $w''=w$. Thus, $\Theta=\Theta''$. 
		
		On the other hand, in view of \cite[Proposition 3.10]{DK2} one sees that $\Lambda''$ is equivalent to $\Lambda'$ and hence $\Theta''$ is equivalent to $\Theta'$ by \cite[Proposition 6.13]{DK2}. So, $\Theta$ is equivalent to $\Theta'$.
	\end{proof}
	
	In particular, when $A$ is commutative, any two equivalent admissible extensions of $A$ by $G$ induce the same partial $G$-module structure on $A$, and a choice of refinements with order-preserving transversals induces a pair of cohomologous partial $2$-cocycles of $G$ with values in this module.
	
	\begin{defn}\label{ext-of-G-mod-by-G-defn}
		Let $(A,\0)$ be a partial $G$-module.  An {\it extension of $(A,\0)$ by $G$} is an admissible extension $A\overset{i}{\to}U\overset{j}{\to}G$ of $A$ by $G$, such that the induced partial $G$-module is $(A,\0)$.
	\end{defn}
	
	\begin{cor}\label{from-cl-adm-ext-to-H^2}
		Each equivalence class $[U]$ of admissible extensions of a partial $G$-module $(A,\0)$ by $G$ determines an element $[w]$ of $H^2(G,A)$.
	\end{cor}
	
	For the converse map we recall from~\cite[Proposition 5.15]{DK2} that any twisted partial action $\Theta$ of $G$ on $A$ defines the admissible extension $A*_\Theta G$ of $A$ by $G$ whose refinement can be chosen to be $A\overset{i}{\to}A*_\Theta G\overset{\pi}{\to}E(A)*_\0 G\overset{\kappa}{\to}G$ with $\pi(a\delta_x)=aa\m\delta_x$ and $\kappa(a\delta_x)=x$.
	
	\begin{lem}\label{Theta'-from-A*_Theta-G}
		Let $\Theta'$ be the twisted partial action induced by $A\overset{i}{\to}A*_\Theta G\overset{\pi}{\to}E(A)*_\0 G\overset{\kappa}{\to}G$ and the order-preserving transversal $\rho:E(A)*_\0 G\to A*_\Theta G$, $\rho(e\delta_x)=e\delta_x$. Then $\Theta'=\Theta$.
	\end{lem}
	\begin{proof}
		By~\cite[Proposition 9.2]{DK2} the twisted partial action $\Theta''$ of $\cG{E(A)*_\0 G}$ on $A$, coming from $A\overset{i}{\to}A*_\Theta G\overset{\pi}{\to}E(A)*_\0 G\overset{\sigma^\natural}{\to}\cG{E(A)*_\0 G}$ and the same $\rho$, satisfies $\Theta=\Theta''\circ\nu$, where $\nu$ is an isomorphism $G\to\cG{E(A)*_\0 G}$ defined by $\nu(x)=\sigma^\natural(e\delta_x)$ for $e\in E(\cD_x)$. Observe that $\nu\circ\kappa(e\delta_x)=\nu(x)=\sigma^\natural(e\delta_x)$, so $\kappa(s)=x\iff\sigma^\natural(s)=\nu(x)$ for any $s\in E(A)*_\Theta G$. Then it is immediately seen from \cref{D-from-Lambda,0-from-Lambda,w-from-Lambda} that $\Theta'=\Theta''\circ\nu$, i.\,e. $\0'_x=\0''_{\nu(x)}$ and $w'_{x,y}=w''_{\nu(x),\nu(y)}$. Thus, $\Theta=\Theta'$.
	\end{proof}
	
	\begin{cor}\label{A*G-ext-G-mod}
		Let $(A,\0)$ be a partial $G$-module and $w$ a twisting related to $(A,\0)$. Then the crossed product $A*_{(\0,w)}G$ is an extension of $(A,\0)$ by $G$.
	\end{cor}
	
	\begin{lem}\label{A*_Lambda-S-eq-A*_Lambda'-S}
		Let $\Lambda=(\alpha,\lambda,f)$ and $\Lambda'=(\alpha',\lambda',f')$ be twisted $S$-module structures on $A$. If $\Lambda$ and $\Lambda'$ are equivalent, then $A*_\Lambda S$ and $A*_{\Lambda'}S$ are equivalent as extensions of $A$ by $S$.
	\end{lem}
	\begin{proof}
		If $g:S\to A$ is a map defining the equivalence, then set $\mu:A*_\Lambda S\to A*_{\Lambda'} S$, $\mu(a\delta_s)=ag(s)\m\delta'_s$. Since $\alpha=\alpha'$ and $g(s)\in A_{\alpha(ss\m)}$, $\mu$ is well-defined. Obviously, $j'\circ\mu=j$. Observe using~\cite[Remark 3.11]{DK2} that
		$$
		\mu\circ i(a)=\mu(a\delta_{\alpha\m(aa\m)})=ag(\alpha\m(aa\m))\m\delta'_{\alpha\m(aa\m)}=a\delta'_{\alpha\m(aa\m)}=i'(a).
		$$
		
		It remains to prove that $\mu$ is a homomorphism. We have by \cref{eq-tw-mod-lambda',eq-tw-mod-f'} and the fact that $g(s)\in A_{\alpha(ss\m)}$
		\begin{align*}
		\mu(a\delta_s)\mu(b\delta_t)&=ag(s)\m\delta'_s\cdot bg(t)\m\delta'_t=ag(s)\m\lambda'_s(bg(t)\m)f'(s,t)\delta'_{st}\\
		&=ag(s)\m g(s)\lambda_s(bg(t)\m)g(s)\m g(s)\lambda_s(g(t))f(s,t)g(st)\m\delta'_{st}\\
		&=a\cdot aa\m\cdot\lambda_s(bg(t)\m)\cdot aa\m\cdot\lambda_s(g(t))f(s,t)g(st)\m\delta'_{st}\\
		&=a\lambda_s(bg(t)\m g(t))f(s,t)g(st)\m\delta'_{st}=a\lambda_s(b\cdot bb\m)f(s,t)g(st)\m\delta'_{st}\\
		&=a\lambda_s(b)f(s,t)g(st)\m\delta'_{st}=\mu(a\delta_s\cdot b\delta_t).
		\end{align*}
	\end{proof}
	
	\begin{lem}\label{eq-tw-to-eq-ext}
		Let $\Theta$ and $\Theta'$ be twisted partial actions of $G$ on $A$. If $\Theta$ is equivalent to $\Theta'$, then $A*_\Theta G$ is equivalent to $A*_{\Theta'}G$.
	\end{lem}
	\begin{proof}
		Observe that $E(A)*_\0 G=E(A)*_{\0'}G=:S$ by~\cite[Remark 5.10]{DK2}. Moreover, the extensions $A\overset{i}{\to}A*_\Theta G\overset{\pi}{\to}S$ and $A\overset{i'}{\to}A*_{\Theta'} G\overset{\pi'}{\to}S$ together with the transversals $\rho:S\to A*_\Theta G$, $\rho(e\delta_x)=e\delta_x$, and $\rho':S\to A*_{\Theta'} G$, $\rho'(e\delta'_x)=e\delta'_x$, induce equivalent twisted $S$-module structures $\Lambda$ and $\Lambda'$ on $A$ thanks to~\cite[Proposition 8.2]{DK2}. Hence $A*_\Lambda S$ is equivalent to $A*_{\Lambda'}S$ by \cref{A*_Lambda-S-eq-A*_Lambda'-S}. But $A*_\Lambda S$ is equivalent to $A*_\Theta G$ and $A*_{\Lambda'}S$ is equivalent to $A*_{\Theta'}G$ by~\cite[Remark 3.18]{DK2}. Thus, by transitivity $A*_\Theta G$ is equivalent to $A*_{\Theta'}G$.
	\end{proof}
	
	\begin{cor}\label{from-H^2-to-cl-adm-ext}
		There is a well-defined map from $H^2(G,A)$ to equivalence classes of extensions of $(A,\0)$ by $G$ which sends a class $[w]$ to the class $[A*_{(\0,w)}G]$, where $w$ is a twisting related to $(A,\0)$ (see \cref{w-cohom-to-twisting}).
	\end{cor}
	
	\begin{thrm}\label{cl-eq-ext<->H^2}
		Let $(A,\0)$ be a partial $G$-module. Then the equivalence classes of extensions of $(A,\0)$ by $G$ are in a one-to-one correspondence with the elements of $H^2(G,A)$.
	\end{thrm}
	\begin{proof}
		If $w$ is a twisting related to $(A,\0)$, then $[w]\mapsto[A*_{(\0,w)}G]\mapsto[w]$ by \cref{Theta'-from-A*_Theta-G}, so $[w]\mapsto[A*_{(\0,w)}G]$ is injective. It is also surjective by~\cite[Theorem 6.12]{DK2}.
	\end{proof}
	
	\section{Split extensions and \texorpdfstring{$H^1(G,A)$}{H\textonesuperior(G,A)}}\label{sec-split}
	
	The classical $H^1(G,A)$ characterizes  (up to $A$-conjugacy) the splittings of the extension $A\rtimes G$ of a $G$-module $A$ by a group $G$  (see~\cite[Proposition IV.2.3]{Brown}). We first introduce a similar notion for an extension of a semilattice of groups by an inverse semigroup.
	
	\subsection{Split extensions of \texorpdfstring{$A$}{A} by \texorpdfstring{$S$}{S}}\label{sec-split-A-by-S}
	
	\begin{defn}\label{split-ext-S-mod-defn}
		An extension $A\overset{i}{\to}U\overset{j}{\to}S$ of $A$ by $S$ is said to {\it split} if there is a transversal $k:S\to U$ of $j$ which is a homomorphism (called a {\it splitting} of $U$).
	\end{defn}
	
	\begin{rem}\label{U'-eq-U-splits}
		If $U$ splits, then any equivalent extension splits.
	\end{rem}
	\noindent For if $\mu:U\to U'$ is an isomorphism determining equivalence and $k:S\to U$ is a splitting of $U$, then $\mu\circ k$ is a splitting of $U'$.
	
	\begin{lem}\label{A*_Lambda-S-splits}
		Let $A$ be a semilattice of abelian groups, $(\alpha,\lambda)$ an $S$-module structure on $A$, $f\in Z^2(S^1,A^1)$ a twisting related to $(\alpha,\lambda)$ and $\Lambda=(\alpha,\lambda,f)$. Then the extension $A*_\Lambda S$ splits if and only if $f\in B^2(S^1,A^1)$.
	\end{lem}
	\begin{proof}
		Observe that any transversal $\rho:S\to A*_\Lambda S$ of $j:A*_\Lambda S\to S$ has the form $\rho(s)=g(s)\delta_s$, where $g(s)\in A_{\alpha(ss\m)}$. Hence, the transversals $\rho$ can be identified with the elements $g$ of $C^1(S^1,A^1)$. Now $\rho$ is a homomorphism if and only if 
		$$
		g(s)\delta_s\cdot g(t)\delta_t=g(st)\delta_{st}\iff g(s)\lambda_s(g(t))f(s,t)=g(st)\iff f=\delta_2^1g.
		$$
	\end{proof}
	
	\begin{lem}\label{U-splits<=>U-eq-A*S}
		An extension $U$ of an $S$-module $A$ by $S$ splits if and only if it is equivalent to $A*_{(\alpha,\lambda)}S$.
	\end{lem}
	\begin{proof}
		Choosing a transversal $\rho$ of $U$, we may assume $U$ to be $A*_{(\alpha,\lambda,f)} S$ for some twisting $f$ related to $(\alpha,\lambda)$. It follows from \cref{A*_Lambda-S-splits} that $U$ splits if and only if $f\in B^2(S^1,A^1)$, that is $f$ is equivalent to the trivial twisting. In view of \cref{A*_Lambda-S-eq-A*_Lambda'-S} the latter exactly means that $U$ is equivalent to $A*_{(\alpha,\lambda)} S$.
	\end{proof}
	
	\begin{lem}\label{splittings<->Z^1(S^1_A^1)}
		Let $U$ be a split extension of an $S$-module $A$ by $S$. Then the splittings of $U$ are in a one-to-one correspondence with the elements of $Z^1(S^1,A^1)$.
	\end{lem}
	\begin{proof}
		Notice that if $U'$ is an equivalent extension with $\mu:U\to U'$ being the corresponding isomorphism, then $k\mapsto\mu\circ k$ defines a bijection between the splittings of $U$ and the splittings of $U'$. Therefore we may assume $U$ to be $A*_{(\alpha,\lambda)} S$ thanks to \cref{U-splits<=>U-eq-A*S}.
		
		Let $k:S\to A*_{(\alpha,\lambda)} S$ be a splitting. As we have seen above $k(s)=g(s)\delta_s$ for some $g\in C^1(S^1,A^1)$. Then
		$$
		g(st)\delta_{st}=k(st)=k(s)k(t)=g(s)\delta_s\cdot g(t)\delta_t=g(s)\lambda_s(g(t))\delta_{st},
		$$
		whence $g(st)=g(s)\lambda_s(g(t))$, that is $g\in Z^1(S^1,A^1)$. 
	\end{proof}
	
	We recall from~\cite[IV.2]{Brown} that in the classical case, given a splitting $k$ of a split group extension of $A$ by $G$ and $a\in A$, the conjugate map $k'(g)=i(a)k(g)i(a)\m$ is again a splitting. This may fail in the semigroup case.
	
	\begin{rem}\label{i(a)k(s)i(a)-inv-splitting}
		Let $k$ be a splitting of a split extension $U$ of $A$ by $S$ and $a\in A$. Then $k'(s)=i(a)k(s)i(a)\m$ is a splitting if and only if $A$ is a monoid (equivalently, $S$ is a monoid, or, equivalently, $U$ is a monoid) and $a\in\cU A$.
	\end{rem}
	\noindent Indeed, observe that $k'(E(S))\subseteq E(U)$, as $k(E(S))\subseteq E(U)$, and $j(k'(s))=j(i(a)k(s)i(a)\m)=j(i(a))sj(i(a))\m$ with $j(i(a))$ being an idempotent. Therefore, $j(k'(s))=s$ for all $s\in S$ if and only if $S$ is a monoid, whose identity is $j(i(a))$. Now
	\begin{align}\label{j(i(a))=alpha-in(aa-inv)}
	j(i(a))=j(i(a))j(i(a))\m=j(i(aa\m))=\alpha\m(aa\m) 
	\end{align}
	by \cref{rho|_E(S)=j|_E(U)-inv,alpha-from-rho}, so $j(i(a))=1_S\iff aa\m=1_A$, i.\,e. $a\in\cU A$. Moreover, in this case $k'$ is a homomorphism, as $i(a)\m i(a)=i(a\m a)=i(1_A)=1_U$.
	
	\begin{defn}\label{A-conjugacy-defn}
		Under the conditions of \cref{i(a)k(s)i(a)-inv-splitting} the splitting $k'$ is said to be {\it $A$-conjugate} to $k$.
	\end{defn}
	
	A more general conjugacy in the non-monoid case is motivated by the following.
	\begin{lem}\label{C^0_<=-conjugate}
		Let $S$ be an inverse semigroup, $U$ a split extension of an $S$-module $A$ by $S$ and $k$ a splitting of $U$. Then for any $h\in C^0_\le(S^1,A^1)$ the map 
		\begin{align}\label{k'=g-conj-of-k}
		k'(s)=i(h(ss\m))k(s)i(h(s\m s))\m  
		\end{align}
		is a splitting of $U$.
	\end{lem}
	\begin{proof}
		By \cref{j(i(a))=alpha-in(aa-inv)} and the fact that $h(e)\in A_{\alpha(e)}$
		$$
		j(k'(s))=\alpha\m(\alpha(ss\m))s\alpha\m(\alpha(s\m s))=ss\m\cdot s\cdot s\m s=s.
		$$
		Moreover, if $e\in E(S)$, then using \cref{alpha-from-rho} we get
		$$
		k'(e)=i(h(e))k(e)i(h(e))\m=i(h(e)\alpha(e)h(e)\m)=i(\alpha(e))=k(e),
		$$
		so $k'(E(S))\subseteq E(U)$. It remains to show that $k'$ is a homomorphism:
		\begin{align}
		k'(s)k'(t)&=i(h(ss\m))k(s)i(h(s\m s)\m h(tt\m))k(t)i(h(t\m t))\m\notag\\
		&=i(h(ss\m))k(s)i(\alpha(s\m s tt\m))k(t)i(h(t\m t))\m \label{k'(s)k'(t)}
		\end{align}
		by \cref{lem-f(e)f(e')-inv}. Since $j(k(s)\m k(s))=s\m s$ and $j(k(t)k(t)\m)=tt\m$, then 
		\begin{align}\label{k(s)k(s)-inv-and-k(s)-inv-k(s)}
		k(s)\m k(s)=i\circ\alpha(s\m s),\  k(t)k(t)\m=i\circ\alpha(tt\m)
		\end{align}
		as $j$ is idempotent-separating. Hence \cref{k'(s)k'(t)} equals 
		\begin{align}\label{i(h(ss-inv))k(st)i(h(t-inv-t))-inv}
		i(h(ss\m))k(s)k(t)i(h(t\m t))\m=i(h(ss\m))k(st)i(h(t\m t))\m.
		\end{align}
		Now applying \cref{k(s)k(s)-inv-and-k(s)-inv-k(s)} to $st$ we rewrite the right-hand side of \cref{i(h(ss-inv))k(st)i(h(t-inv-t))-inv} as 
		\begin{align}\label{i(alpha(stt-inv-s-inv)h(ss-inv))k(st)i(alpha(t-inv-s-inv-st)h(t-inv-t))-inv}
		i(\alpha(stt\m s\m)h(ss\m))k(st)i(\alpha(t\m s\m st)h(t\m t))\m.
		\end{align}
		Observe that $\alpha(stt\m s\m)h(ss\m)=h(stt\m s\m)$, as $stt\m s\m\le ss\m$ and $h$ is order-preserving. Similarly $\alpha(t\m s\m st)h(t\m t)=h(t\m s\m st)$, so \cref{i(alpha(stt-inv-s-inv)h(ss-inv))k(st)i(alpha(t-inv-s-inv-st)h(t-inv-t))-inv} is $k'(st)$.
	\end{proof}
	
	\begin{defn}\label{conjugate-splittings-defn}
		Under the conditions of \cref{C^0_<=-conjugate} the splittings $k$ and $k'$ are said to be {\it $C^0_\le$-equivalent}.
	\end{defn}
	
	\begin{rem}\label{g-conjugation-monoid-case}
		Observe that in the monoid case \cref{k'=g-conj-of-k} becomes
		$$
		k'(s)=i(h(1_S))i(\alpha(ss\m))k(s)i(\alpha(s\m s))i(h(1_S))\m=i(h(1_S))k(s)i(h(1_S))\m.
		$$
		by \cref{f(ee')=alpha(e')f(e),k(s)k(s)-inv-and-k(s)-inv-k(s)}. Since $h(1_S)\in \cU A$, the $C^0_\le$-equivalence generalizes the $A$-conjugacy.
	\end{rem}
	
	\begin{prop}\label{conjugate-classes<->H^1_<=(S^1_A^1)}
		Under the conditions of \cref{splittings<->Z^1(S^1_A^1)} the $C^0_\le$-equivalence classes of splittings of $U$ are in a one-to-one correspondence with the elements of $H^1_\le(S^1,A^1)$.
	\end{prop}
	\begin{proof}
		We first observe that the $C^0_\le$-equivalence agrees with the equivalence of extensions. Indeed, if $\mu:U\to U'$ is an isomorphism respecting the diagrams of $U$ and $U'$, then $k'(s)=i(h(ss\m))k(s)i(h(s\m s))\m$ exactly when $\mu(k'(s))=i'(h(ss\m))\mu(k(s))i'(h(s\m s))\m$. So, $k'$ is $C^0_\le$-equivalent to $k$ if and only if $\mu\circ k'$ is $C^0_\le$-equivalent to $\mu\circ k$. This shows that we may assume $U$ to be $A*_{(\alpha,\lambda)}S$ as in the proof of \cref{splittings<->Z^1(S^1_A^1)}.
		
		Let $k$ and $k'$ be splittings of $A*_{(\alpha,\lambda)}S$, and $g,g'$ the corresponding elements of $Z^1(S^1,A^1)$ (see \cref{splittings<->Z^1(S^1_A^1)}). In view of \cref{H^1_le(S^1_A^1)-cong-H^1(S^1_A^1)} it is enough to prove that $k'(s)=i(h(ss\m))k(s)i(h(s\m s))\m$ if and only if $g'(s)=g(s)\lambda_s(p(s\m s))p(ss\m)\m$ for some $p\in C^0_\le(S^1,A^1)$.
		
		Using \cref{a-delta_s-inv} and \cref{lambda_e(a)} of the definition of a twisted $S$-module, one has
		\begin{align*}
		g'(s)\delta_s&=k'(s)=i(h(ss\m))k(s)i(h(s\m s))\m\\
		&=h(ss\m)\delta_{ss\m}\cdot g(s)\delta_s\cdot (h(s\m s)\delta_{s\m s})\m\\
		&=h(ss\m)\lambda_{ss\m}(g(s))\delta_s\cdot \lambda_{s\m s}(h(s\m s)\m)\delta_{s\m s}\\
		&=h(ss\m)g(s)\delta_s\cdot h(s\m s)\m\delta_{s\m s}\\
		&=h(ss\m)g(s)\lambda_s(h(s\m s)\m)\delta_s,
		\end{align*}
		so we may take $p(e)=h(e)\m$.
	\end{proof}
	
	\begin{rem}\label{A-conjugate-classes<->H^1(S_A)}
		Under the conditions of \cref{splittings<->Z^1(S^1_A^1)} assume that $S$ is a monoid. Then the $A$-conjugacy classes of splittings of $U$ are in a one-to-one correspondence with the elements of $H^1(S,A)$.
	\end{rem}
	\noindent  This can be easily explained using \cref{H^1(S_A)}.
	
	\subsection{Split extensions of \texorpdfstring{$A$}{A} by \texorpdfstring{$G$}{G}}\label{sec-split-A-by-G}
	
	Let $A\overset{i}{\to}U\overset{j}{\to}G$ be an extension of $(A,\0)$ by $G$ and $A\overset{i}{\to}U\overset{\pi}{\to}S\overset{\kappa}{\to}G$ a refinement of $U$. We recall from~\cite[Proposition 7.4]{DK2} that there is a one-to-one correspondence between the transversals of $\pi$ and the partial maps $\tau:G\times E(U)\dashrightarrow U$, such that 
	\begin{enumerate}
		\item $\tau(x,e)$ is defined $\iff U(x,e)=\{u\in U\mid j(u)=x,\ \ uu\m=e\}\ne\emptyset$;
		\item $\tau(x,e)\in U(x,e)$, whenever defined;
		\item $\tau(1,e)=e$.
	\end{enumerate}
	Such a map $\tau$ was called a {\it transversal} of $j$ in~\cite{DK2}. 
	
	More precisely, given a transversal $\rho$ of $\pi$, the corresponding transversal $\tau$ of $j$ is defined by 
	\begin{align}\label{tau-from-rho}
	\tau(x,e)=\rho\circ\pi(u),
	\end{align}
	where $u$ is an arbitrary element of $U(x,e)$. The definition does not depend on $u$, since the sets $U(x,e)$ are exactly the classes of $\ker\pi$ by~\cite[Lemma 7.1]{DK2}. Conversely, by $\tau$ one constructs 
	\begin{align}\label{rho-from-tau}
	\rho(s)=\tau(\kappa(s),\pi\m(ss\m)),
	\end{align}
	which is a transversal of $\pi$.
	
	\begin{lem}\label{rho-splitting-in-terms-of-tau}
		The transversal $\rho$ is a splitting of $U$ if and only if the corresponding $\tau$ satisfies
		\begin{align}\label{tau(x_e)tau(y_f)=tau(xy_tau(x_e)-f-tau(x_e)-inv)}
		\tau(x,e)\tau(y,f)=\tau(xy,\tau(x,e)f\tau(x,e)\m)
		\end{align}
		for all $x,y\in G$ and $e,f\in E(U)$  (the equality should be understood as follows: if the left-hand side is defined, then the right-hand side is defined and they coincide).
	\end{lem}
	\begin{proof}
		Let $\rho:S\to U$ be a splitting of $U$. Suppose that $\tau(x,e)$ and $\tau(y,f)$ are defined. This means that $U(x,e)$ and $U(y,f)$ are nonempty, so that $\tau(x,e)=\rho\circ\pi(u)$ and $\tau(y,f)=\rho\circ\pi(v)$ for some $u\in U(x,e)$ and $v\in U(y,f)$. Then by \cref{tau-from-rho}
		\begin{align*}
		\tau(x,e)\tau(y,f)=\rho(\pi(u))\rho(\pi(v))=\rho\circ\pi(uv),
		\end{align*}
		as $\rho$ is a homomorphism. Take $w=\tau(x,e)v$ and note that
		\begin{align*}
		j(w)&=j(\tau(x,e)v)=j(\tau(x,e))j(v)=xy,\\
		ww\m&=\tau(x,e)vv\m\tau(x,e)\m=\tau(x,e)f\tau(x,e)\m,
		\end{align*}
		hence $w\in U(xy,\tau(x,e)f\tau(x,e)\m)\ne\emptyset$. Moreover,
		\begin{align*}
		\tau(xy,\tau(x,e)f\tau(x,e)\m)=\rho\circ\pi(w)=\rho(\pi(\tau(x,e))\pi(v))=\rho(\pi(u)\pi(v))=\rho\circ\pi(uv),
		\end{align*}
		proving \cref{tau(x_e)tau(y_f)=tau(xy_tau(x_e)-f-tau(x_e)-inv)}. Here we used the fact that $\pi(\tau(x,e))=\pi(u)$, since both $u$ and $\tau(x,e)$ belong to the same $U(x,e)$.
		
		Conversely, assume that \cref{tau(x_e)tau(y_f)=tau(xy_tau(x_e)-f-tau(x_e)-inv)} holds. For $s,t\in S$ one has 
		\begin{align*}
		\rho(s)\rho(t)&=\tau(\kappa(s),\pi\m(ss\m))\tau(\kappa(t),\pi\m(tt\m)),\\
		\rho(st)&=\tau(\kappa(st),\pi\m(stt\m s\m))=\tau(\kappa(s)\kappa(t),\pi\m(stt\m s\m))
		\end{align*}
		by \cref{rho-from-tau}. In view of \cref{tau(x_e)tau(y_f)=tau(xy_tau(x_e)-f-tau(x_e)-inv)}  in order to prove that $\rho(s)\rho(t)=\rho(st)$, it remains to show that 
		\begin{align}\label{pi(stt-inv-s-inv)=tau-pi(tt-inv)-tau-inv}
		{\pi}\m(stt\m s\m)=\tau(\kappa(s),\pi\m(ss\m))\pi\m(tt\m)\tau(\kappa(s),\pi\m(ss\m))\m.
		\end{align}
		But $\tau(\kappa(s),\pi\m(ss\m))=\rho(s)$, and hence the right-hand side of \cref{pi(stt-inv-s-inv)=tau-pi(tt-inv)-tau-inv} is mapped to $stt\m s\m$ by $\pi$. So, \cref{pi(stt-inv-s-inv)=tau-pi(tt-inv)-tau-inv} follows from the fact that $\pi$ separates idempotents.
	\end{proof}
	
	\begin{defn}\label{A->U->G-splits-defn}
		We say that an extension $A\overset{i}{\to}U\overset{j}{\to}G$ of $(A,\0)$ by $G$ {\it splits}, if there exists a transversal $\tau$ of $j$ satisfying \cref{tau(x_e)tau(y_f)=tau(xy_tau(x_e)-f-tau(x_e)-inv)}. In this case $\tau$ is called a {\it splitting} of $U$.
	\end{defn}
	
	\begin{rem}\label{U-splits-iff-a-refinement-splits}
		Observe that $A\overset{i}{\to}U\overset{j}{\to}G$ splits if and only if there is a refinement $A\overset{i}{\to}U\overset{\pi}{\to}S\overset{\kappa}{\to}G$, such that $A\overset{i}{\to}U\overset{\pi}{\to}S$ splits.
	\end{rem}
	\noindent This is explained by~\cite[Proposition 7.4]{DK2} and \cref{rho-splitting-in-terms-of-tau}.
	
	In particular, any split extension is automatically admissible.
	
	\begin{rem}\label{equiv-ext-splits}
		If an extension $U$ of $(A,\0)$ by $G$ splits, then any equivalent extension splits.
	\end{rem}
	\noindent For if $U'$ is equivalent to $U$, $\mu:U\to U'$ is the corresponding isomorphism and $\tau$ is a splitting of $U$, then $\tau'(x,e')=\mu\circ\tau(x,\mu\m(e'))$, where $x\in G$ and $e'\in E(U')$, is a splitting of $U'$.
	
	\begin{prop}\label{split_ext_is_A*_0-G}
		An extension of $(A,\0)$ by $G$ splits if and only if it is equivalent to $A*_\0 G$.
	\end{prop}
	\begin{proof}
		By \cref{cl-eq-ext<->H^2} any extension $U$ of $(A,\0)$ by $G$ is equivalent to $A*_{(\0,w)}G$ for some twisting $w$ related to $(A,\0)$. In view of \cref{equiv-ext-splits} $U$ splits if and only if $A*_{(\0,w)}G$ does. Consider the ``standard'' refinement $A\overset{i}{\to}A*_{(\0,w)}G\overset{\pi}{\to}S\overset{\kappa}{\to}G$, where $S=E(A)*_\0 G$. Let $\Lambda=(\alpha,\lambda,f)$ be the induced $S$-module structure on $A$ together with the corresponding order-preserving twisting related to it (see \cref{lambda-from-Theta,f-from-Theta}). By~\cite[Theorem 7.4]{Lausch} $A*_{(\0,w)}G$ and $A*_\Lambda S$ are equivalent as extensions of $A$ by $S$. Therefore, $A*_\Lambda S$ also splits, so $f$ is a coboundary thanks to \cref{A*_Lambda-S-splits}. It follows from \cref{H^n(GA)-cong-H^n_<=(SA)} that $w$ is a partial coboundary, and thus $(\0,w)$ is equivalent to $\0$ with the trivial twisting. Then the extension $A*_{(\0,w)}G$ is equivalent to $A*_\0 G$ by \cref{eq-tw-to-eq-ext}.
	\end{proof} 
	
	\begin{lem}\label{splittings<->Z^1(G_A)}
		Let $A\overset{i}{\to}U\overset{j}{\to}G$ be a split extension of $(A,\0)$ by $G$. Then there is a one-to-one correspondence between the splittings of $U$ and the elements of $Z^1(G,A)$.
	\end{lem}
	\begin{proof}
		Taking into account \cref{equiv-ext-splits,split_ext_is_A*_0-G} we may assume that $U=A*_\0 G$. Moreover, we shall choose the standard refinement $A\overset{i}{\to}U\overset{\pi}{\to}S\overset{\kappa}{\to}G$ of $U$, where $S=E(A)*_\0 G$. 
		
		By \cref{rho-splitting-in-terms-of-tau} and~\cite[Proposition 7.4]{DK2} the splittings of $A\overset{i}{\to}U\overset{j}{\to}G$ are in a one-to-one correspondence with the splittings of $A\overset{i}{\to}U\overset{\pi}{\to}S$, which are in a one-to-one correspondence with the elements of $Z^1(S^1,A^1)$ by \cref{splittings<->Z^1(S^1_A^1)}. Observe that $Z^1(S^1,A^1)=Z^1_\le(S^1,A^1)$ thanks to \cref{H^1_le(S^1_A^1)-cong-H^1(S^1_A^1)}. Furthermore, the $S$-module structure on $A$ coming from $A\overset{i}{\to}U\overset{\pi}{\to}S$ is exactly the one defined by \cref{lambda-from-Theta}  thanks to~\cite[Lemma 8.1]{DK2}. Therefore, $Z^1_\le(S^1,A^1)\cong Z^1(G,A)$ by \cref{H^n(GA)-cong-H^n_<=(SA)}.
	\end{proof}
	
	\begin{lem}\label{k-equiv-k'-in-terms-of-tau}
		Two splittings $k$ and $k'$ of $A\overset{i}{\to}U\overset{\pi}{\to}S$ are $C^0_\le$-equivalent if and only if there exists an order-preserving $\eta:E(U)\to A$ with $\eta(e)\in A_{i\m(e)}$, such that the corresponding splittings $\tau$ and $\tau'$ of $A\overset{i}{\to}U\overset{j}{\to}G$ satisfy
		\begin{align}\label{tau-equiv-tau'}
		\tau'(x,e)=i(\eta(e))\tau(x,e)i(\eta(\tau(x,e)\m e\tau(x,e)))\m.
		\end{align}
	\end{lem}
	\begin{proof}
		Suppose that $k$ and $k'$ are $C^0_\le$-equivalent, that is, there is $h\in C^0_\le(S^1,A^1)$, such that $k'(s)=i(h(ss\m))k(s)i(h(s\m s))\m$. Set $\eta=h\circ\pi|_{E(U)}$ and observe that $\eta:E(U)\to A$ is order-preserving, as $h$ is, and $\eta(e)\in A_{\alpha\circ\pi(e)}=A_{i\m(e)}$. If $\tau'(x,e)$ is defined, then taking $u\in U(x,e)\ne\emptyset$, one has by \cref{tau-from-rho}
		\begin{align*}
		\tau'(x,e)=k'\circ\pi(u)&=i(h(\pi(uu\m)))k(\pi(u))i(h(\pi(u\m u)))\m.
		\end{align*}
		But $uu\m=e$, so $i(h(\pi(uu\m)))=i(\eta(e))$. Moreover,
		$$
		\pi(u\m u)=\pi(u\m eu)=\pi(u)\m\pi(e)\pi(u)=\pi(\tau(x,e))\m\pi(e)\pi(\tau(x,e)),
		$$
		whence $u\m u=\tau(x,e)\m e\tau(x,e)$, proving \cref{tau-equiv-tau'}.
		
		Conversely, if \cref{tau-equiv-tau'} holds, then setting $h=\eta\circ\pi\m|_{E(S)}$ and using \cref{rho-from-tau} we obtain
		\begin{align*}
		k'(s)=\tau'(\kappa(s),\pi\m(ss\m))&=i(h(ss\m))k(s)i(\eta(k(s)\m \pi\m(ss\m)k(s))))\m\\
		&=i(h(ss\m))k(s)i(h(s\m s))\m,
		\end{align*}
		as $k|_{E(S)}=\pi\m|_{E(S)}$. Clearly $h\in C^0_\le(S^1,A^1)$, so $k'$ is $C^0_\le$-equivalent to $k$.
	\end{proof}
	
	\begin{defn}\label{defn-A-conjugacy}
		Two splittings $\tau$ and $\tau'$ of a split extension $A\overset{i}{\to}U\overset{j}{\to}G$ are said to be {\it equivalent}, if they satisfy \cref{tau-equiv-tau'} for some order-preserving $\eta:E(U)\to A$ with $\eta(e)\in A_{i\m(e)}$.
	\end{defn}
	
	\begin{thrm}\label{eq-classes-of-splittings<->H^1(G_A)}
		Under the conditions of \cref{splittings<->Z^1(G_A)} the equivalence classes of splittings of $U$ are in a one-to-one correspondence with the elements of $H^1(G,A)$.
	\end{thrm}
	\begin{proof}
		As in the proof of \cref{splittings<->Z^1(G_A)} choose $U$ to be $A*_\0 G$ and consider the standard refinement $A\overset{i}{\to}U\overset{\pi}{\to}S\overset{\kappa}{\to}G$ of $U$. By \cref{k-equiv-k'-in-terms-of-tau} there is a one-to-one correspondence between the equivalence classes of splittings of $A\overset{i}{\to}U\overset{j}{\to}G$ and the $C^0_\le$-equivalence classes of splittings of $A\overset{i}{\to}U\overset{\pi}{\to}S$. It remains to apply \cref{conjugate-classes<->H^1_<=(S^1_A^1),H^n(GA)-cong-H^n_<=(SA)}.
	\end{proof}
	
	\section*{Acknowledgements}
	
	We thank the referee for useful comments.
	
	\bibliography{bibl-pact}{}
	\bibliographystyle{acm}
	
\end{document}